\documentclass[11pt]{amsart}
\usepackage{amssymb,amsmath,amsthm}
\usepackage[margin=1.2in]{geometry}
\usepackage[colorlinks=true,linkcolor=blue,citecolor=blue,urlcolor=blue]{hyperref}

\newtheorem{theorem}{Theorem}[section]
\newtheorem{lemma}[theorem]{Lemma}
\newtheorem{proposition}[theorem]{Proposition}
\newtheorem{corollary}[theorem]{Corollary}
\theoremstyle{definition}
\newtheorem{definition}[theorem]{Definition}
\newtheorem{question}[theorem]{Question}
\theoremstyle{remark}
\newtheorem{remark}[theorem]{Remark}

\newcommand{\Cantor}{2^{\omega}}
\newcommand{\strings}{2^{<\omega}}
\newcommand{\ltt}{\le_{\mathrm{tt}}}
\newcommand{\ett}{\equiv_{\mathrm{tt}}}
\newcommand{\Dtt}{\mathcal{D}_{\mathrm{tt}}}
\newcommand{\cat}{{}^{\frown}}
\newcommand{\res}{\upharpoonright}
\newcommand{\dg}{\deg_{\mathrm{tt}}}
\newcommand{\Padd}{\Phi_{+}}
\newcommand{\Psub}{\Phi_{-}}
\newcommand{\Phalf}{\Phi_{1/2}}
\newcommand{\Radd}{\textup{(R$_{+}$)}}
\newcommand{\Rsub}{\textup{(R$_{-}$)}}
\newcommand{\Rhalf}{\textup{(R$_{1/2}$)}}

\title[Baire measurable functions and automorphisms of the truth-table degrees]{Baire measurable functions induce only the trivial automorphism of the truth-table degrees}

\author{Bj{\o}rn Kjos-Hanssen}

\begin{document}

\begin{abstract}
We show that if an automorphism $\pi$ of the truth-table degrees is induced by a
Baire measurable function $F\colon\Cantor\to\Cantor$, or merely induced on some
comeager set by a continuous function, then $\pi$ is the identity.
The same
results hold for the weak truth-table degrees.
The argument is an adaptation to
truth-table reducibility of OpenAI's recent proof that the
Turing degrees are rigid.
\end{abstract}

\maketitle

\section{Introduction}

Let $\Dtt$ denote the partial order of the truth-table degrees. An
\emph{automorphism} of $\Dtt$ is a bijection $\pi\colon\Dtt\to\Dtt$ such that
$\mathbf{x}\le\mathbf{y}\iff\pi(\mathbf{x})\le\pi(\mathbf{y})$ for all
$\mathbf{x},\mathbf{y}\in\Dtt$. Let $D\subseteq\Cantor$ and
$F\colon D\to\Cantor$. We say that $F$ \emph{induces $\pi$ on $D$} if
\[
  \dg(F(X))=\pi(\dg(X))\qquad\text{for all }X\in D ,
\]
and that $F$ \emph{induces} $\pi$ if this holds with $D=\Cantor$.

The truth-table degrees sit between two structures whose automorphism groups
are at opposite extremes. The many-one degrees have $2^{2^{\omega}}$
automorphisms (Shore; see the introduction of \cite{And}). The Turing degrees
have only the trivial one: this is the theorem recently announced in
\cite{OAI}, an exposition of which is given by Marks \cite{Mar}. The proof in
\cite{Mar} combines two ingredients. The first is the theorem of Slaman and
Woodin \cite{SW} that every automorphism of the Turing degrees is induced by an
arithmetic function. The second is a coding lemma \cite[Lemma~0.2]{Mar}, proved
with Cohen forcing and the group structure of the $2$-adic integers, which
shows that an automorphism induced by such a function is the identity.

For the truth-table degrees much less is known. Anderson \cite{And} showed that
every automorphism of $\Dtt$ is the identity on a cone; see also \cite{KH3}.
In this note we show that the second ingredient above adapts to truth-table
reducibility, in a strong form.

\begin{theorem}\label{thm:main}
  Let $\pi$ be an automorphism of $\Dtt$. Suppose that there are a comeager set
  $D\subseteq\Cantor$ and a continuous function $F\colon D\to\Cantor$ which
  induces $\pi$ on $D$. Then $\pi$ is the identity.
\end{theorem}

Recall that $F\colon\Cantor\to\Cantor$ is \emph{Baire measurable} if preimages
of open sets have the Baire property, and that such a function is continuous on
a dense $G_\delta$ set \cite[8.38]{Kec}. Every Borel function is Baire
measurable.

\begin{corollary}\label{cor:bm}
  If an automorphism $\pi$ of $\Dtt$ is induced by a Baire measurable function
  $F\colon\Cantor\to\Cantor$, then $\pi$ is the identity.
\end{corollary}

For functions which are continuous on all of $\Cantor$ there is a short
self-contained proof, with some additional information.

\begin{theorem}\label{thm:cts}
  Suppose that an automorphism $\pi$ of $\Dtt$ is induced by a continuous
  function $F\colon\Cantor\to\Cantor$. Then $\pi$ is the identity, that is,
  $F(X)\ett X$ for all $X$. Moreover, every nonempty open subset of $\Cantor$
  contains a basic open set $N_\tau$ on which $F$ coincides with a total
  computable functional and has a total computable left inverse: there are
  total computable $\Lambda,\Gamma\colon\Cantor\to\Cantor$ with
  $F(Z)=\Lambda(Z)$ and $\Gamma(F(Z))=Z$ for all $Z\in N_\tau$.
\end{theorem}

Theorem~\ref{thm:cts} is of course a special case of Theorem~\ref{thm:main} as
far as triviality of $\pi$ is concerned, but its proof (Section~\ref{sec:cts})
uses neither forcing nor any descriptive set theory beyond the Baire category
theorem for open subsets of Cantor space. For comparison, before \cite{OAI}
the corresponding statement for the Turing degrees had been obtained for
special classes of continuous $F$: those induced by a permutation of $\omega$
\cite{KH1}, and homeomorphisms which are bi-uniform $E_0$-isomorphisms
\cite{KH2}. For the truth-table degrees no hypothesis beyond continuity is
needed.

The hypothesis of Theorem~\ref{thm:main} is satisfied by every automorphism if
all sets of reals are sufficiently regular. We work this out for the Axiom of
Determinacy, $\mathrm{AD}$, in Section~\ref{sec:ad}.

\begin{theorem}[$\mathrm{ZF}+\mathrm{AD}$]\label{thm:ad}
  Every automorphism of $\Dtt$ is the identity.
\end{theorem}

Everything else in this note is proved in $\mathrm{ZFC}$.

Finally, the proofs use only a few properties of truth-table reductions, which
we isolate in Section~\ref{sec:other}. They are shared by weak truth-table
reducibility.

\begin{theorem}\label{thm:wttintro}
  Theorems~\ref{thm:main}, \ref{thm:cts} and~\ref{thm:ad} and
  Corollary~\ref{cor:bm} hold for the wtt-degrees in place of $\Dtt$.
\end{theorem}

This is Corollary~\ref{cor:wtt} below.

\subsection*{How the proof differs from the Turing case}
Truth-table reducibility is at first sight less suited to the argument of
\cite{Mar} than Turing reducibility, because the decoding procedure of
\cite[Lemma~0.2]{Mar} is presented as an unbounded search, and a procedure with
an unbounded search is not a truth-table reduction. Two observations remove the
difficulty, and a third removes the need for any hypothesis on the inverse
automorphism.

\begin{enumerate}
\item A truth-table reduction is given by a \emph{total} computable functional
  $\Phi\colon\Cantor\to\Cantor$, which is continuous. Hence a statement of the
  form ``$\varphi(x)=\Phi(\psi(x))$'', for continuous $\varphi$ and $\psi$,
  defines a \emph{closed} set of parameters $x$. So the Baire category theorem
  directly yields an open set of parameters on which a single $\Phi$ works
  (Lemma~\ref{lem:unif}). In the continuous case this replaces forcing
  altogether; in the general case it replaces forcing by a statement about
  comeager sets.
\item The search in \cite{Mar} is in fact bounded. There one has a real $X'$
  extending $0^{l}$ and searches for an $n$ such that $G+nX'$ extends a given
  string $q_1$ of length $l+d$; such an $n$ exists if and only if $X'(l)=1$.
  But if $X'(l)=1$ then $X'/2^{l}$ is a $2$-adic unit, so some $n<2^{d}$
  works. Thus $X'(l)=1$ if and only if the search succeeds for some $n<2^{d}$,
  and the second, parallel search of \cite{Mar} is not needed. In the
  continuous case one can even arrange $d=1$, so that there is no search at
  all. (Below we organize the bounded search slightly differently, using
  translations by finitely many fixed integers instead of multiples of $X'$.)
\item In \cite{Mar} the coding lemma is applied to a function inducing
  $\pi^{-1}$, which is available there because the theorem of Slaman and Woodin
  applies to every automorphism. Here we are given a function for $\pi$ only.
  The coding lemma for $\pi$ gives $\pi^{-1}(\mathbf{z})\le\mathbf{z}$ for all
  $\mathbf{z}$, so that representatives of $\pi^{-1}(\dg(Z))$ are computed from
  $Z$ by tt-functionals. To find a single tt-functional that works on a
  comeager set we need the set of $Z$ on which a given functional works to have
  the Baire property. This holds because the degrees of the reals in any
  comeager set generate $\Dtt$ under joins and meets (compare Jockusch and
  Posner \cite{JP} for the Turing degrees), so that $\pi$ is
  determined by $F\res D$ in a way that makes its graph Borel
  (Lemma~\ref{lem:graph}).
\end{enumerate}

With the first two points, every step of the decoding becomes a total
computable functional, and a uniformly computable sequence of total functionals
can be iterated without leaving the class of truth-table reductions.

\subsection*{What remains open}
Since the identity automorphism is induced by the identity function,
Theorem~\ref{thm:main} and Corollary~\ref{cor:bm} show that the rigidity of
$\Dtt$ is \emph{equivalent} to a positive answer to either part of the
following question.

\begin{question}\label{q:rep}
  \begin{enumerate}
  \item Is every automorphism of $\Dtt$ induced on some comeager set by a
    continuous function?
  \item Is every automorphism of $\Dtt$ induced by an arithmetic function, as
    in the theorem of Slaman and Woodin \cite{SW} for the Turing degrees?
  \end{enumerate}
\end{question}

By Theorem~\ref{thm:ad} the answer is positive under $\mathrm{AD}$. In
$\mathrm{ZFC}$ the question is open as far as we know.

\section{Preliminaries}\label{sec:prelim}

\subsection{Notation}
We write $\Cantor$ for Cantor space and $\strings$ for the set of finite binary
strings. For $\sigma\in\strings$, $|\sigma|$ is its length and
$N_\sigma=\{X\in\Cantor:\sigma\subset X\}$ is the corresponding basic clopen
set. For $\sigma\in\strings$ and $X\in\Cantor$, $\sigma\cat X$ is the
concatenation; $0^{k}$ is the string of $k$ zeros and $0^{\omega}$ the
all-zero sequence. $X\res n$ is the initial segment of $X$ of length $n$. The
shift $S\colon\Cantor\to\Cantor$ is given by $S(X)(n)=X(n+1)$, so that
$S(e\cat X)=X$ for $e\in\{0,1\}$. $X\oplus Y$ is the usual recursive join. For
bits $b,e\in\{0,1\}$, $b\oplus e$ also denotes addition modulo $2$; no
confusion should arise.

Let $U$ be a topological space. We say that a property holds \emph{for comeager
many} $u\in U$ if the set of $u\in U$ with the property is comeager in $U$.

\subsection{Truth-table functionals}
A \emph{tt-functional} is a total computable function
$\Phi\colon\Cantor\to\Cantor$; that is, $\Phi$ is computed by an oracle Turing
machine which halts on every oracle and every input. As is well known (see
e.g.\ \cite{Odi}), for $A,B\in\Cantor$,
\[
  A\ltt B\iff A=\Phi(B)\text{ for some tt-functional }\Phi .
\]
We use the following evident facts without further comment. There are only
countably many tt-functionals. Every tt-functional is continuous. The
composition of two tt-functionals is a tt-functional; if $\Phi$ is a
tt-functional and $P\in\Cantor$ is computable then $A\mapsto\Phi(A\oplus P)$ is
a tt-functional; a function defined by cases, according to the value of a fixed
bit of the argument, from tt-functionals is a tt-functional. Finally, if
$\Theta$ and $\Psi$ are tt-functionals and $j\in\omega$, then $A\mapsto
(\Psi(\Theta^{m}(A))(j))_{m\in\omega}$ is a tt-functional, where $\Theta^{m}$
is the $m$-fold iterate of $\Theta$: to compute the $m$th bit, run the (always
halting) computations of $\Theta$ successively $m$ times and then that of
$\Psi$.

The truth-table degree of $X$ is denoted $\dg(X)$. The degree $\dg(A\oplus B)$
is the least upper bound of $\dg(A)$ and $\dg(B)$ in $\Dtt$, and the degree of
the computable sets is the least element of $\Dtt$.

\subsection{The \texorpdfstring{$2$}{2}-adic group structure}
As in \cite{Mar}, we identify $X\in\Cantor$ with the $2$-adic integer
$\sum_{n}X(n)2^{n}$, and we write $+$ and $-$ for $2$-adic addition and
subtraction (addition of binary expansions with carries propagating to the
right). This makes $\Cantor$ a compact abelian topological group with identity
$0^{\omega}$. A natural number $n$ is identified with the element of $\Cantor$
given by its binary expansion followed by zeros; in particular
$2^{l}=0^{l}\cat 1\cat 0^{\omega}$. We collect the facts we need.

\begin{lemma}\label{lem:2adic}
  \begin{enumerate}
  \item\label{2a:tt} The maps $(X,Y)\mapsto X+Y$ and $X\mapsto -X$ are
    computable and total; in particular $X+Y\ltt X\oplus Y$ and $-X\ltt X$.
    For each $n$, $(X+Y)\res n$ depends only on $X\res n$ and $Y\res n$: it is
    the binary expansion of the sum of the two integers modulo $2^{n}$.
  \item\label{2a:sub} For each $k$, $N_{0^{k}}$ is a subgroup. If $T\in
    N_{0^{k}}$ then $(X+T)\res k=X\res k$ for every $X$. Consequently, for
    every $\alpha\in\strings$ with $|\alpha|\le k$ and every $T\in N_{0^{k}}$,
    the translation $X\mapsto X+T$ is a homeomorphism of $N_\alpha$ onto
    itself.
  \item\label{2a:half} $2X=0\cat X$. For $T\in N_{0^{1}}$ put $T/2=S(T)$; then
    $2(T/2)=T$, the map $T\mapsto T/2$ is computable on $N_{0^1}$, and
    $T\in N_{0^{k+1}}$ implies $T/2\in N_{0^{k}}$.
  \item\label{2a:prefix} Let $\tau\in\strings$ have length $l$ and let
    $W\in\Cantor$. Then $\tau\cat 0^{\omega}+0^{l}\cat W=\tau\cat W$.
  \end{enumerate}
\end{lemma}

\begin{proof}
  All of this is immediate from the definition of addition with carries.
  For~\eqref{2a:sub}, $N_{0^k}$ is the set of $2$-adic integers divisible by
  $2^{k}$; the translation by $T\in N_{0^{k}}$ maps $N_\alpha$ into itself, and
  so does its inverse, the translation by $-T\in N_{0^{k}}$.
  For~\eqref{2a:prefix}, no carry occurs.
\end{proof}

\subsection{Basic properties of a function inducing an automorphism}

\begin{lemma}\label{lem:basic}
  Suppose that $F\colon\Cantor\to\Cantor$ induces an automorphism $\pi$ of
  $\Dtt$. Then for all $A,B\in\Cantor$:
  \begin{itemize}
  \item[\textup{(M)}] if $A\ltt B$ then $F(A)\ltt F(B)$;
  \item[\textup{(J)}] $F(A\oplus B)\ett F(A)\oplus F(B)$;
  \item[\textup{(C)}] if $A$ is computable then $F(A)$ is computable;
  \item[\textup{(I)}] if $F(A)\ett F(B)$ then $A\ett B$.
  \end{itemize}
  Consequently:
  \begin{itemize}
  \item[\textup{(N)}] $F$ is not constant on any uncountable set; in
    particular $F$ is not constant on any nonempty open set, nor on any set
    which is comeager in a nonempty open set.
  \end{itemize}
\end{lemma}

\begin{proof}
  (M) and (I) hold because $\pi$ is an order isomorphism. (J) and (C) hold
  because an automorphism of a partial order preserves least upper bounds and
  the least element. For (N): a truth-table degree is a countable set of reals,
  so an uncountable set contains $A,B$ with $A\not\ett B$, and then
  $F(A)\neq F(B)$ by (I). A set which is comeager in a nonempty open subset of
  $\Cantor$ is uncountable.
\end{proof}

In Section~\ref{sec:general} the function $F$ is only defined on a comeager
set, and we use the following version of Lemma~\ref{lem:basic}.

\begin{lemma}\label{lem:basicD}
  Let $\pi$ be an automorphism of $\Dtt$, let $D\subseteq\Cantor$, and suppose
  that $F\colon D\to\Cantor$ induces $\pi$ on $D$.
  \begin{enumerate}
  \item\label{bd:mj} Let $B,C_1,\dots,C_n\in D$, where $n\ge 1$. If
    $B\ltt C_1\oplus\dots\oplus C_n$, then
    $F(B)\ltt F(C_1)\oplus\dots\oplus F(C_n)$. More generally, let
    $X,A\in\Cantor$ with $\dg(A)=\pi(\dg(X))$. If
    $B\ltt X\oplus C_1\oplus\dots\oplus C_n$, then
    $F(B)\ltt A\oplus F(C_1)\oplus\dots\oplus F(C_n)$.
  \item\label{bd:inj} If $B,C\in D$ and $F(B)\ett F(C)$, then $B\ett C$.
    Consequently $F$ is not constant on any uncountable subset of $D$.
  \end{enumerate}
\end{lemma}

\begin{proof}
  \eqref{bd:mj} Since $\pi$ preserves order and finite joins,
  $\dg(F(B))=\pi(\dg(B))\le\pi(\dg(X)\vee\dg(C_1)\vee\dots\vee\dg(C_n))
  =\pi(\dg(X))\vee\pi(\dg(C_1))\vee\dots\vee\pi(\dg(C_n))
  =\dg(A\oplus F(C_1)\oplus\dots\oplus F(C_n))$. The first statement is proved
  in the same way. \eqref{bd:inj} holds because $\pi$ is injective and
  truth-table degrees are countable.
\end{proof}

\subsection{Uniformization by Baire category}

The following lemma is the point at which the totality of truth-table
reductions is used.

\begin{lemma}[Uniformization]\label{lem:unif}
  Let $\mathcal{S}$ be a nonempty Baire space and let
  $\varphi,\psi\colon\mathcal{S}\to\Cantor$ be continuous functions such that
  $\varphi(x)\ltt\psi(x)$ for every $x\in\mathcal{S}$. Then there are a
  nonempty open set $V\subseteq\mathcal{S}$ and a tt-functional $\Phi$
  such that
  \[
    \varphi(x)=\Phi(\psi(x))\qquad\text{for all }x\in V .
  \]
\end{lemma}

\begin{proof}
  For each tt-functional $\Phi$ let
  $C_\Phi=\{x\in\mathcal{S}:\varphi(x)=\Phi(\psi(x))\}$. Since $\varphi$ and
  $\Phi\circ\psi$ are continuous maps into a Hausdorff space, $C_\Phi$ is
  closed in $\mathcal{S}$. By hypothesis $\mathcal{S}=\bigcup_\Phi C_\Phi$, a
  countable union. As $\mathcal{S}$ is a nonempty Baire space, some $C_\Phi$ is
  not nowhere dense, and being closed it has nonempty interior $V$.
\end{proof}

Recall that compact Hausdorff spaces and Polish spaces are Baire spaces, that
a nonempty open subset of a Baire space is a Baire space, and that a dense
$G_\delta$ subset of a Polish space is Polish \cite[3.11, 8.4]{Kec}. We shall
apply Lemma~\ref{lem:unif} only to $\Cantor$, to $\Cantor\times\Cantor$, to
dense $G_\delta$ subsets of these spaces, and to nonempty relatively open
subsets of such sets.

\subsection{The three relations}

Fix a set $D\subseteq\Cantor$, a function $F\colon D\to\Cantor$ and
tt-functionals $\Padd,\Psub,\Phalf$. (In Section~\ref{sec:cts} we will have
$D=\Cantor$.) Let
\[
  D^{*}=\bigl\{(G,H)\in\Cantor\times\Cantor:\ G,\ H,\ G+H,\ -H,\ G+2H\in D
  \bigr\}.
\]
For a pair $(G,H)\in D^{*}$ we consider the three relations
\begin{align*}
  \Radd &\qquad F(G+H)=\Padd\bigl(F(G)\oplus F(H)\bigr),\\
  \Rsub &\qquad F(G)=\Psub\bigl(F(G+H)\oplus F(-H)\bigr),\\
  \Rhalf &\qquad F(G+H)=\Phalf\bigl(F(G)\oplus F(G+2H)\bigr).
\end{align*}
The corresponding reductions between the arguments of $F$ are
\begin{equation}\label{eq:args}
  G+H\ltt G\oplus H,\qquad G\ltt (G+H)\oplus(-H),\qquad
  G+H\ltt G\oplus(G+2H),
\end{equation}
which hold by Lemma~\ref{lem:2adic}: for the last one,
$G+(G+2H)=2(G+H)$, so $G+H=(G+(G+2H))/2$. Hence, if $F$ has the property that
\begin{equation}\label{eq:MJ}
  B\ltt C_1\oplus C_2\ \text{ implies }\ F(B)\ltt F(C_1)\oplus F(C_2)
  \qquad\text{for all }B,C_1,C_2\in D,
\end{equation}
then for every pair $(G,H)\in D^{*}$,
\begin{equation}\label{eq:Fargs}
  \begin{aligned}
    F(G+H)&\ltt F(G)\oplus F(H),\\
    F(G)&\ltt F(G+H)\oplus F(-H),\\
    F(G+H)&\ltt F(G)\oplus F(G+2H).
  \end{aligned}
\end{equation}
Property \eqref{eq:MJ} holds if $D=\Cantor$ and $F$ satisfies (M) of
Lemma~\ref{lem:basic} and $F(A\oplus B)\ltt F(A)\oplus F(B)$ for all $A,B$
(which is half of (J)); and it holds if $F$ induces an automorphism on $D$, by
Lemma~\ref{lem:basicD}\eqref{bd:mj}.

The next lemma records the two consequences of the three relations that drive
the decoding. They correspond to items (3) and (2) of the Claim in
\cite{Mar}.

\begin{lemma}[Translation and halving]\label{lem:trans}
  Let $E\subseteq D^{*}$ be a set of pairs each of which
  satisfies \Radd, \Rsub\ and \Rhalf. Let $G,H,T,\kappa\in\Cantor$.
  \begin{enumerate}
  \item\label{tr:trans} If $(G+T,\,H+\kappa)\in E$ and $(G+T+\kappa,\,H)\in E$,
    then
    \[
      F(G+T+\kappa)=\Psub\Bigl(\Padd\bigl(F(G+T)\oplus F(H+\kappa)\bigr)\oplus
      F(-H)\Bigr).
    \]
  \item\label{tr:half} If $T\in N_{0^{1}}$, $(G+T,\,H-T/2)\in E$ and
    $(G+T/2,\,H)\in E$, then
    \[
      F(G+T/2)=\Psub\Bigl(\Phalf\bigl(F(G+T)\oplus F(G+2H)\bigr)\oplus
      F(-H)\Bigr).
    \]
  \end{enumerate}
\end{lemma}

\begin{proof}
  Note that since $E\subseteq D^{*}$, all arguments of $F$ occurring below lie
  in $D$.

  \eqref{tr:trans} By \Radd\ for the pair $(G+T,H+\kappa)$,
  $F(G+T+H+\kappa)=\Padd(F(G+T)\oplus F(H+\kappa))$. By \Rsub\ for the pair
  $(G+T+\kappa,H)$, $F(G+T+\kappa)=\Psub(F(G+T+\kappa+H)\oplus F(-H))$.

  \eqref{tr:half} Let $G'=G+T$ and $H'=H-T/2$. Then $G'+H'=G+H+T/2$ and
  $G'+2H'=G+T+2H-2(T/2)=G+2H$. By \Rhalf\ for the pair $(G',H')$,
  $F(G+H+T/2)=\Phalf(F(G+T)\oplus F(G+2H))$. By \Rsub\ for the pair
  $(G+T/2,H)$, $F(G+T/2)=\Psub(F(G+T/2+H)\oplus F(-H))$.
\end{proof}

\section{The continuous case}\label{sec:cts}

Throughout this section $F\colon\Cantor\to\Cantor$ is continuous. The main step
is the following proposition. It uses only properties (M) and (C) of
Lemma~\ref{lem:basic} and weak forms (J$'$), (N$'$) of (J) and (N), and not
the full strength of the assumption that $F$ induces an automorphism.

\begin{proposition}\label{prop:lower}
  Let $F\colon\Cantor\to\Cantor$ be continuous and suppose that
  \begin{itemize}
  \item[\textup{(M)}] $A\ltt B$ implies $F(A)\ltt F(B)$,
  \item[\textup{(J$'$)}] $F(A\oplus B)\ltt F(A)\oplus F(B)$,
  \item[\textup{(C)}] $F(A)$ is computable whenever $A$ is computable, and
  \item[\textup{(N$'$)}] $F$ is not constant on any nonempty open set.
  \end{itemize}
  Then for every $\rho\in\strings$ there are a string $\tau\supseteq\rho$ and a
  tt-functional $\Gamma$ such that
  \[
    \Gamma\bigl(F(\tau\cat X)\bigr)=X\qquad\text{for all }X\in\Cantor.
  \]
  In particular, $X\ltt F(X)$ for all $X\in\Cantor$.
\end{proposition}

\begin{proof}
  The last sentence follows from the first: $\tau\cat X\ltt X$, so
  $X\ltt F(\tau\cat X)\ltt F(X)$ by (M). We prove the first assertion. Fix
  $\rho\in\strings$.

  \medskip
  \noindent\emph{Step 1: uniform relations on an open set.}
  Here $D=\Cantor$ and $D^{*}=\Cantor\times\Cantor$. By (M) and (J$'$), $F$
  has property \eqref{eq:MJ}, so the reductions \eqref{eq:Fargs} hold for
  every pair $(G,H)$. The maps $(G,H)\mapsto F(G+H)$, $F(G)\oplus F(H)$,
  $F(G)$, $F(G+H)\oplus F(-H)$ and $F(G)\oplus F(G+2H)$ are continuous on
  $\Cantor\times\Cantor$, since $F$ and the group operations are. We apply
  Lemma~\ref{lem:unif} three times; each time the space is a nonempty open
  subset of the compact space $\Cantor\times\Cantor$, hence a Baire space.
  First, on the space $N_\rho\times N_\rho$, with $\varphi(G,H)=F(G+H)$ and
  $\psi(G,H)=F(G)\oplus F(H)$, we obtain a nonempty open
  $V_1\subseteq N_\rho\times N_\rho$ and a tt-functional $\Padd$ such that
  \Radd\ holds for all pairs in $V_1$. Second, on the space $V_1$, with
  $\varphi(G,H)=F(G)$ and $\psi(G,H)=F(G+H)\oplus F(-H)$, we obtain a nonempty
  open $V_2\subseteq V_1$ and $\Psub$ such that \Rsub\ holds for all pairs in
  $V_2$. Third, on the space $V_2$, with $\varphi(G,H)=F(G+H)$ and
  $\psi(G,H)=F(G)\oplus F(G+2H)$, we obtain a nonempty open
  $V_3\subseteq V_2$ and $\Phalf$ such that \Rhalf\ holds for all pairs in
  $V_3$. Choose $k$ and $\alpha,\beta\in\strings$ with
  $|\alpha|=|\beta|=k$ and $N_\alpha\times N_\beta\subseteq V_3$. Then
  $\alpha\supseteq\rho$, and
  \begin{equation}\label{eq:E}
    \text{every pair in }E:=N_\alpha\times N_\beta
    \text{ satisfies \Radd, \Rsub\ and \Rhalf.}
  \end{equation}
  By Lemma~\ref{lem:2adic}\eqref{2a:sub}, $E$ is invariant under translation by
  elements of $N_{0^{k}}\times N_{0^{k}}$.

  \medskip
  \noindent\emph{Step 2: a splitting bit.}
  By (N$'$), $F$ is not constant on $N_\alpha$, so there is $j\in\omega$ such
  that the function $f\colon\Cantor\to\{0,1\}$, $f(Z)=F(Z)(j)$, is not constant
  on $N_\alpha$. Since $f$ is continuous and $\Cantor$ is compact, there is
  $n_0$ such that $f(Z)$ depends only on $Z\res n_0$. Thus $f$ is constant on
  $N_\nu$ whenever $|\nu|\ge n_0$. Let $\tau\supseteq\alpha$ be a string of maximal
  length such that $f$ is not constant on $N_\tau$; such a string exists
  because $\alpha$ has this property and every string with this property has
  length less than $n_0$. By maximality, $f$ is constant on $N_{\tau\cat 0}$ and
  on $N_{\tau\cat 1}$, and the two constant values are different since $f$ is
  not constant on $N_\tau=N_{\tau\cat 0}\cup N_{\tau\cat 1}$. Let $l=|\tau|$,
  so $l\ge k$, and let $b$ be the value of $f$ on $N_{\tau\cat 0}$. Then
  \begin{equation}\label{eq:read}
    F(\tau\cat W)(j)=b\oplus W(0)\qquad\text{for all }W\in\Cantor .
  \end{equation}

  \medskip
  \noindent\emph{Step 3: the shift functional.}
  Let
  \[
    G=\tau\cat 0^{\omega},\qquad H=\beta\cat 0^{\omega},\qquad Y=2^{l}
    =0^{l}\cat 1\cat 0^{\omega}.
  \]
  These are computable, and so are $H-Y$, $-H$ and $G+2H$. By (C) the reals
  \[
    P_1=F(H-Y),\qquad P_2=F(-H),\qquad P_3=F(G+2H)
  \]
  are computable. By Lemma~\ref{lem:2adic}\eqref{2a:prefix},
  $G+0^{l}\cat W=\tau\cat W$ for every $W$.

  We claim that there is a tt-functional $\Theta$ such that
  \begin{equation}\label{eq:shift}
    \Theta\bigl(F(\tau\cat W)\bigr)=F(\tau\cat S(W))
    \qquad\text{for all }W\in\Cantor .
  \end{equation}
  Define, for $A\in\Cantor$,
  \[
    A'=\begin{cases}
      A & \text{if }A(j)=b,\\
      \Psub\bigl(\Padd(A\oplus P_1)\oplus P_2\bigr) & \text{if }A(j)\ne b,
    \end{cases}
  \]
  and
  \[
    \Theta(A)=\Psub\bigl(\Phalf(A'\oplus P_3)\oplus P_2\bigr).
  \]
  Then $\Theta$ is a tt-functional, since $P_1,P_2,P_3$ are computable. To
  verify \eqref{eq:shift}, fix $W$, let $A=F(\tau\cat W)$ and
  \[
    T=0^{l}\cat W,\qquad T'=0^{l+1}\cat S(W).
  \]
  So $T,T'\in N_{0^{l}}\subseteq N_{0^{k}}$, $A=F(G+T)$, $T'\in N_{0^{l+1}}$
  and $T'/2=0^{l}\cat S(W)\in N_{0^l}$. If $W(0)=0$ then $T'=T$; if $W(0)=1$
  then $T=T'+Y$ (adding $2^{l}$ to $T'$ changes its bit $l$ from $0$ to $1$,
  with no carry), that is, $T'=T-Y$.

  We first show that $A'=F(G+T')$. If $W(0)=0$ then $A(j)=b$ by
  \eqref{eq:read}, so $A'=A=F(G+T)=F(G+T')$. If $W(0)=1$ then $A(j)\neq b$ by
  \eqref{eq:read}, so $A'=\Psub(\Padd(F(G+T)\oplus F(H-Y))\oplus F(-H))$. We
  apply Lemma~\ref{lem:trans}\eqref{tr:trans} with $\kappa=-Y$: the pairs
  $(G+T,H-Y)$ and $(G+T-Y,H)$ belong to $E=N_\alpha\times N_\beta$, because
  $G\in N_\alpha$, $H\in N_\beta$ and $T,\,-Y,\,T-Y\in N_{0^{k}}$. Hence
  $A'=F(G+T-Y)=F(G+T')$.

  Next we apply Lemma~\ref{lem:trans}\eqref{tr:half} to $T'$: the pairs
  $(G+T',H-T'/2)$ and $(G+T'/2,H)$ belong to $E$, because $T'$ and $T'/2$ lie
  in $N_{0^{k}}$. Hence
  \[
    \Theta(A)=\Psub\bigl(\Phalf(F(G+T')\oplus F(G+2H))\oplus F(-H)\bigr)
    =F(G+T'/2)=F(\tau\cat S(W)),
  \]
  which is \eqref{eq:shift}.

  \medskip
  \noindent\emph{Step 4: decoding.}
  Define $\Gamma\colon\Cantor\to\Cantor$ by
  $\Gamma(A)(m)=b\oplus\Theta^{m}(A)(j)$. This is a tt-functional. Let
  $X\in\Cantor$. By \eqref{eq:shift} and induction on $m$,
  $\Theta^{m}(F(\tau\cat X))=F(\tau\cat S^{m}(X))$, so by \eqref{eq:read},
  \[
    \Gamma\bigl(F(\tau\cat X)\bigr)(m)=b\oplus F(\tau\cat S^{m}(X))(j)
    =b\oplus b\oplus S^{m}(X)(0)=X(m).\qedhere
  \]
\end{proof}

\begin{remark}\label{rem:dichotomy}
  Hypothesis (N$'$) can be replaced by a dichotomy. Let $F$ be continuous and
  satisfy (M), (J$'$) and (C). If (N$'$) fails, say $F$ is constant on
  $N_\sigma$ with value $c$, then $c=F(\sigma\cat 0^{\omega})$ is computable by
  (C), and for every $X$ we have $X\ltt\sigma\cat X$, so $F(X)\ltt
  F(\sigma\cat X)=c$ by (M). Thus either $F(X)$ is computable for every $X$,
  or $X\ltt F(X)$ for every $X$.
\end{remark}

\begin{proposition}\label{prop:inverse}
  Suppose that a continuous function $F$ induces an automorphism $\pi$ of
  $\Dtt$. Then $\pi^{-1}$ is induced by a tt-functional.
\end{proposition}

\begin{proof}
  By Lemma~\ref{lem:basic}, $F$ satisfies the hypotheses of
  Proposition~\ref{prop:lower}, so $X\ltt F(X)$ for all $X$.

  Let $Z\in\Cantor$. Since $\pi$ is onto, there is $X$ with
  $\dg(X)=\pi^{-1}(\dg(Z))$, that is, $F(X)\ett Z$. Then $X\ltt F(X)\ltt Z$.
  Hence there are tt-functionals $\Phi,\Lambda,\Lambda'$ with $X=\Phi(Z)$,
  $\Lambda(F(\Phi(Z)))=Z$ and $\Lambda'(Z)=F(\Phi(Z))$. In other words,
  $\Cantor=\bigcup_{\Phi,\Lambda,\Lambda'}C_{\Phi,\Lambda,\Lambda'}$, where
  \[
    C_{\Phi,\Lambda,\Lambda'}=\bigl\{Z\in\Cantor:\Lambda(F(\Phi(Z)))=Z
    \text{ and }\Lambda'(Z)=F(\Phi(Z))\bigr\}.
  \]
  Each of these countably many sets is closed, because $F$ and the
  tt-functionals are continuous. By the Baire category theorem there are
  $\Phi,\Lambda,\Lambda'$ and $\sigma\in\strings$ with
  $N_\sigma\subseteq C_{\Phi,\Lambda,\Lambda'}$. Define
  $\widetilde F(Z)=\Phi(\sigma\cat Z)$. This is a tt-functional, and for every
  $Z$ we have $F(\widetilde F(Z))\ett\sigma\cat Z\ett Z$, as witnessed by
  $\Lambda$ and $\Lambda'$. Thus $\pi(\dg(\widetilde F(Z)))=\dg(Z)$, i.e.\
  $\dg(\widetilde F(Z))=\pi^{-1}(\dg(Z))$.
\end{proof}

\begin{proof}[Proof of Theorem~\ref{thm:cts}]
  Let $\widetilde F$ be a tt-functional inducing $\pi^{-1}$, by
  Proposition~\ref{prop:inverse}. Then $\widetilde F(Z)\ltt Z$ for all $Z$,
  trivially. On the other hand $\widetilde F$ is continuous and induces the
  automorphism $\pi^{-1}$, so Lemma~\ref{lem:basic} and
  Proposition~\ref{prop:lower} apply to $\widetilde F$ and give
  $Z\ltt\widetilde F(Z)$ for all $Z$. Hence $\widetilde F(Z)\ett Z$ for all
  $Z$, so $\pi^{-1}$ is the identity, and so is $\pi$. Thus $F(X)\ett X$ for
  all $X$.

  For the ``moreover'' part, let $O$ be a nonempty open set and choose
  $\rho$ with $N_\rho\subseteq O$. Since $F(X)\ltt X$ for all $X$,
  Lemma~\ref{lem:unif}, applied on the space $N_\rho$ with $\varphi=F$ and
  $\psi$ the identity, gives $\rho'\supseteq\rho$ and a tt-functional $\Lambda$ with $F=\Lambda$ on
  $N_{\rho'}$. By Proposition~\ref{prop:lower} applied to $\rho'$ there are
  $\tau\supseteq\rho'$ and a tt-functional $\Gamma_0$ with
  $\Gamma_0(F(\tau\cat X))=X$ for all $X$. Let
  $\Gamma(A)=\tau\cat\Gamma_0(A)$. Then $F(Z)=\Lambda(Z)$ and $\Gamma(F(Z))=Z$
  for all $Z\in N_\tau$.
\end{proof}

\begin{remark}
  In Theorem~\ref{thm:cts} the function $F$ is not assumed to be injective or
  surjective, and need not be either. The shift $S$ is continuous, induces the
  identity on $\Dtt$, and is two-to-one; the map $X\mapsto 0\cat X$ is
  continuous, induces the identity on $\Dtt$, and is not surjective. What the
  theorem does give is local injectivity on a dense open set: every nonempty
  open set contains a basic open set $N_\tau$ on which $F$ is injective, with
  an inverse $F[N_\tau]\to N_\tau$ that is the restriction of a tt-functional.
\end{remark}

\begin{corollary}\label{cor:perm}
  Let $h\colon\omega\to\omega$ be a permutation and let
  $F_h\colon\Cantor\to\Cantor$ be the homeomorphism $F_h(X)=X\circ h$. Then
  $F_h$ induces an automorphism of $\Dtt$ if and only if $h$ is computable, and
  in that case the automorphism is the identity.
\end{corollary}

\begin{proof}
  If $h$ is computable then so is $h^{-1}$, and $X\circ h\ett X$ for all $X$.
  Conversely, suppose that $F_h$ induces an automorphism. By
  Theorem~\ref{thm:cts} there are $\tau\in\strings$ and a tt-functional
  $\Lambda$ with $\Lambda(Z)(n)=Z(h(n))$ for all $Z\in N_\tau$ and all $n$.
  Only finitely many $n$ have $h(n)<|\tau|$. For any other $n$, we compute
  $h(n)$ as follows: run the computation of $\Lambda(\tau\cat 0^{\omega})(n)$,
  let $Q$ be the finite set of positions $q\ge|\tau|$ of the oracle that are
  queried, and for each $q\in Q$ compute $\Lambda(Z_q)(n)$, where $Z_q$ is
  $\tau\cat 0^{\omega}$ with bit $q$ changed to $1$. Since
  $\Lambda(Z)(n)=Z(h(n))$ on $N_\tau$, the value $\Lambda(Z_q)(n)$ is $1$ if
  $q=h(n)$ and $0$ otherwise; and $h(n)\in Q$, since otherwise the computation
  of $\Lambda(\tau\cat 0^{\omega})(n)$ would give the same answer $0$ on the
  oracle $Z_{h(n)}$, whereas $\Lambda(Z_{h(n)})(n)=1$. So $h$ is computable.
\end{proof}

Corollary~\ref{cor:perm} is the truth-table analogue of the main results of
\cite{KH1} for the Turing degrees.

\section{The general case}\label{sec:general}

We now prove Theorem~\ref{thm:main}. The core is a coding lemma
(Proposition~\ref{prop:coding}), the truth-table analogue of
\cite[Lemma~0.2]{Mar}. In place of arithmetic Cohen generics we use comeager
sets, and in place of ``forcing equals truth'' we use Lemma~\ref{lem:unif}.

Throughout Sections~\ref{ss:E}--\ref{ss:inverse} we fix the following data:
\begin{equation}\label{eq:data}
  \parbox{0.82\textwidth}{$\pi$ is an automorphism of $\Dtt$, $D\subseteq\Cantor$
  is a dense $G_\delta$ set, and $F\colon D\to\Cantor$ is a continuous function
  which induces $\pi$ on $D$.}
\end{equation}
The set $D^{*}$ and the relations \Radd, \Rsub, \Rhalf\ are as in
Section~\ref{sec:prelim}.

\subsection{Uniform relations on a comeager set}\label{ss:E}

\begin{lemma}\label{lem:E}
  Assume \eqref{eq:data}. Then there are $k\in\omega$, strings $\alpha,\beta$
  of length $k$, tt-functionals $\Padd,\Psub,\Phalf$ and a set
  $E\subseteq D^{*}\cap(N_\alpha\times N_\beta)$ such that
  \begin{enumerate}
  \item $E$ is comeager in $N_\alpha\times N_\beta$;
  \item every pair in $E$ satisfies \Radd, \Rsub\ and \Rhalf.
  \end{enumerate}
\end{lemma}

\begin{proof}
  Each of the five maps $(G,H)\mapsto G$, $H$, $G+H$, $-H$, $G+2H$ from
  $\Cantor\times\Cantor$ to $\Cantor$ is continuous and open: each is the
  composition of a homeomorphism of $\Cantor\times\Cantor$ (respectively the
  identity, $(G,H)\mapsto(H,G)$, $(G,H)\mapsto(G+H,H)$,
  $(G,H)\mapsto(-H,G)$, $(G,H)\mapsto(G+2H,H)$) with the projection to the
  first coordinate. The preimage of a dense open set under a continuous open
  map is dense and open, so the preimage of the dense $G_\delta$ set $D$ under
  each of the five maps is a dense $G_\delta$. Hence $D^{*}$ is a dense
  $G_\delta$ subset of $\Cantor\times\Cantor$; in particular it is a Polish
  space, and so are its nonempty relatively open subsets.

  On $D^{*}$, the maps $(G,H)\mapsto F(G)$, $F(H)$, $F(G+H)$, $F(-H)$,
  $F(G+2H)$ are continuous, being compositions of continuous maps into $D$ with
  $F$. By Lemma~\ref{lem:basicD}\eqref{bd:mj}, $F$ has property \eqref{eq:MJ},
  so the reductions \eqref{eq:Fargs} hold for all pairs in $D^{*}$. As in
  Step~1 of the proof of Proposition~\ref{prop:lower}, three applications of
  Lemma~\ref{lem:unif}, on the Baire space $D^{*}$ and then on nonempty
  relatively open subsets of it, give nonempty relatively open sets
  $V_3\subseteq V_2\subseteq V_1\subseteq D^{*}$ and tt-functionals
  $\Padd,\Psub,\Phalf$ such that all pairs in $V_3$ satisfy \Radd, \Rsub\ and
  \Rhalf. Write $V_3=D^{*}\cap O$ with $O$ open in $\Cantor\times\Cantor$, and
  choose $k$ and $\alpha,\beta$ of length $k$ with
  $N_\alpha\times N_\beta\subseteq O$ (note that $O\neq\emptyset$). Let
  $E=D^{*}\cap(N_\alpha\times N_\beta)$. Then $E\subseteq V_3$, and $E$ is
  comeager in $N_\alpha\times N_\beta$ because $D^{*}$ is comeager in
  $\Cantor\times\Cantor$.
\end{proof}

\subsection{A splitting bit}

\begin{lemma}\label{lem:split}
  Assume \eqref{eq:data} and let $\alpha\in\strings$. Then there are
  $j\in\omega$, $b,i\in\{0,1\}$, a string $\sigma\supseteq\alpha$ of some
  length $l$, a number $d\ge 1$ and a string $q_1\supseteq\sigma\cat(1-i)$ of
  length $l+d$ such that
  \begin{align*}
    F(Z)(j)&=b &&\text{for all }Z\in D\cap N_{\sigma\cat i},\\
    F(Z)(j)&=1-b &&\text{for all }Z\in D\cap N_{q_1}.
  \end{align*}
\end{lemma}

\begin{proof}
  The set $D\cap N_\alpha$ is comeager in $N_\alpha$, hence uncountable, so by
  Lemma~\ref{lem:basicD}\eqref{bd:inj}, $F$ is not constant on it. Fix $j$ such
  that $f(Z)=F(Z)(j)$ is not constant on $D\cap N_\alpha$. Call a string $\nu$
  \emph{decided} if $f$ is constant on $D\cap N_\nu$. Every extension of a
  decided string is decided, and $\alpha$ is not decided; hence no initial
  segment of $\alpha$ is decided.

  Fix $Z_0\in D\cap N_\alpha$. Since $f$ is continuous at $Z_0$ (as a function
  on $D$) and takes values in a discrete space, some initial segment of $Z_0$
  is decided. Let $q_0$ be the shortest decided initial segment of $Z_0$. Then
  $|q_0|>|\alpha|$, so $q_0=\sigma\cat i$ where $\sigma\supseteq\alpha$ is not
  decided and $i\in\{0,1\}$. Let $l=|\sigma|$ and let $b$ be the constant value
  of $f$ on $D\cap N_{q_0}$ (this set is nonempty, as it contains $Z_0$).

  Since $\sigma$ is not decided, there is $Z_1\in D\cap N_\sigma$ with
  $f(Z_1)=1-b$. Then $Z_1\notin N_{q_0}$, so $Z_1\supseteq\sigma\cat(1-i)$. By
  continuity of $f$ at $Z_1$ there is an initial segment $q_1$ of $Z_1$, which
  we may take of length greater than $l$, such that $f$ has the constant value
  $1-b$ on $D\cap N_{q_1}$. Let $d=|q_1|-l$.
\end{proof}

\subsection{The coding lemma}

\begin{proposition}[Coding lemma]\label{prop:coding}
  Assume \eqref{eq:data}. Then there is a nonempty open set
  $U\subseteq\Cantor\times\Cantor$ with the following property. For every
  $X\in\Cantor$ there is a set $\mathcal{G}_X\subseteq U\cap(D\times D)$,
  comeager in $U$, such that for every $A\in\Cantor$ with
  $\dg(A)=\pi(\dg(X))$,
  \[
    X\ltt A\oplus F(G)\oplus F(H)\qquad\text{for all }(G,H)\in\mathcal{G}_X.
  \]
\end{proposition}

\begin{proof}
  Let $k$, $\alpha$, $\beta$, $\Padd$, $\Psub$, $\Phalf$ and $E$ be as in
  Lemma~\ref{lem:E}, and let $j$, $b$, $i$, $\sigma$, $l$, $d$ and $q_1$ be as
  in Lemma~\ref{lem:split} for this $\alpha$. So $l\ge k$. Let
  $q_0=\sigma\cat i$ and
  \[
    U=N_{q_0}\times N_\beta .
  \]
  Let $Y=2^{l}$ and, for $0\le c<2^{d-1}$, let $\kappa_c=2^{l+1}c$, an element
  of $N_{0^{l+1}}$ (recall that natural numbers are identified with elements of
  $\Cantor$).

  Fix $X\in\Cantor$. For $m\in\omega$ let
  \[
    X_m=0^{l}\cat S^{m}(X)\in N_{0^{l}},
  \]
  and let $C_X$ be the subgroup of $N_{0^{l}}$ generated by
  $\{X_m:m\in\omega\}\cup\{Y\}$. Then $C_X$ is countable, and
  $\kappa_c=2c\cdot Y\in C_X$ for each $c$.

  \medskip
  \noindent\emph{The good pairs.} Let
  \[
    \mathcal{G}_X=\bigl\{(G,H)\in U:\ (G+c_1,\,H+c_2)\in E\text{ for all }
    c_1,c_2\in C_X\bigr\}.
  \]
  We claim that $\mathcal{G}_X$ is comeager in $U$. Fix $c_1,c_2\in C_X$.
  Since $c_1,c_2\in N_{0^{l}}\subseteq N_{0^{k}}$, the translation
  $t(G,H)=(G+c_1,H+c_2)$ is a homeomorphism of $N_\alpha\times N_\beta$ onto
  itself by Lemma~\ref{lem:2adic}\eqref{2a:sub}. As $E$ is comeager in
  $N_\alpha\times N_\beta$, so is $t^{-1}(E)$, and therefore $t^{-1}(E)\cap U$
  is comeager in the open subset $U$ of $N_\alpha\times N_\beta$. The set
  $\mathcal{G}_X$ is the intersection of these countably many sets
  $t^{-1}(E)\cap U$, which proves the claim. Taking $c_1=c_2=0$ and recalling
  that $E\subseteq D^{*}$, we see that $\mathcal{G}_X\subseteq D\times D$.

  \medskip
  For the rest of the proof fix $(G,H)\in\mathcal{G}_X$ and $A$ with
  $\dg(A)=\pi(\dg(X))$; we show that $X\ltt A\oplus F(G)\oplus F(H)$. Since
  $E\subseteq D^{*}$, for all $c_1,c_2\in C_X$ the reals
  \begin{equation}\label{eq:inD}
    G+c_1,\qquad H+c_2,\qquad -H,\qquad G+2H
  \end{equation}
  belong to $D$ (for the last two use the pair $(G,H)\in E$). For $T\in C_X$
  write
  \[
    a(T)=F(G+T).
  \]

  \medskip
  \noindent\emph{Translation.} For all $T,\kappa\in C_X$,
  \begin{equation}\label{eq:btrans}
    a(T+\kappa)=\Psub\Bigl(\Padd\bigl(a(T)\oplus F(H+\kappa)\bigr)\oplus
    F(-H)\Bigr).
  \end{equation}
  This is Lemma~\ref{lem:trans}\eqref{tr:trans}: the pairs $(G+T,H+\kappa)$ and
  $(G+T+\kappa,H)$ are in $E$ because $(G,H)\in\mathcal{G}_X$ and
  $T,\kappa,T+\kappa\in C_X$.

  \medskip
  \noindent\emph{Halving.} For all $T\in C_X\cap N_{0^{l+1}}$ such that
  $T/2\in C_X$,
  \begin{equation}\label{eq:bhalf}
    a(T/2)=\Psub\Bigl(\Phalf\bigl(a(T)\oplus F(G+2H)\bigr)\oplus F(-H)\Bigr).
  \end{equation}
  This is Lemma~\ref{lem:trans}\eqref{tr:half}: the pairs $(G+T,H-T/2)$ and
  $(G+T/2,H)$ are in $E$ because $T,\,-T/2,\,T/2\in C_X$.

  \medskip
  \noindent\emph{Reading a bit.} We claim that for every $T\in C_X$,
  \begin{equation}\label{eq:bread}
    T(l)=1\iff\text{there is }c<2^{d-1}\text{ such that }
    a(T+\kappa_c)(j)=1-b .
  \end{equation}
  Recall that $T\in N_{0^{l}}$ and $G\supseteq q_0=\sigma\cat i$. Let $s<2^{l}$
  be the natural number whose binary expansion is $\sigma$.

  Suppose first that $T(l)=0$, i.e.\ $T\in N_{0^{l+1}}$. Then for every $c$,
  $T+\kappa_c\in N_{0^{l+1}}$, so $G+T+\kappa_c\in N_{q_0}$ by
  Lemma~\ref{lem:2adic}\eqref{2a:sub}. Also $G+T+\kappa_c\in D$ by
  \eqref{eq:inD}. Hence $a(T+\kappa_c)(j)=b$ for every $c$, by
  Lemma~\ref{lem:split}.

  Suppose now that $T(l)=1$. Modulo $2^{l+1}$ we have $G\equiv s+i2^{l}$ and
  $T\equiv 2^{l}$, so $G+T\equiv s+(1-i)2^{l}$; that is,
  $G+T\supseteq\sigma\cat(1-i)$. Hence there is $r<2^{d-1}$ with
  \[
    G+T\equiv s+(1-i)2^{l}+2^{l+1}r\pmod{2^{l+d}} .
  \]
  Similarly, since $q_1\supseteq\sigma\cat(1-i)$ has length $l+d$, the natural
  number with binary expansion $q_1$ is $s+(1-i)2^{l}+2^{l+1}r_1$ for some
  $r_1<2^{d-1}$. Let $c<2^{d-1}$ with $c\equiv r_1-r\pmod{2^{d-1}}$. Then
  $2^{l+1}(r+c)\equiv 2^{l+1}r_1\pmod{2^{l+d}}$, so
  \[
    G+T+\kappa_c\equiv s+(1-i)2^{l}+2^{l+1}r_1\pmod{2^{l+d}},
  \]
  that is, $G+T+\kappa_c\in N_{q_1}$, by Lemma~\ref{lem:2adic}\eqref{2a:tt}.
  Also $G+T+\kappa_c\in D$. Hence $a(T+\kappa_c)(j)=1-b$ by
  Lemma~\ref{lem:split}. This proves \eqref{eq:bread}.

  \medskip
  \noindent\emph{The update functional.} Let $P$ be the join of the finitely
  many reals
  \[
    F(H+\kappa_c)\ (c<2^{d-1}),\qquad F(H-Y),\qquad F(-H),\qquad F(G+2H),
  \]
  which are defined by \eqref{eq:inD}. We define two functions of a real $Z$,
  uniformly computable from $P$. First let
  \[
    \varepsilon(Z)=\begin{cases}
      1 & \text{if }\Psub\bigl(\Padd(Z\oplus F(H+\kappa_c))\oplus
      F(-H)\bigr)(j)=1-b\text{ for some }c<2^{d-1},\\
      0 & \text{otherwise,}
    \end{cases}
  \]
  then let
  \[
    Z'=\begin{cases}
      Z & \text{if }\varepsilon(Z)=0,\\
      \Psub\bigl(\Padd(Z\oplus F(H-Y))\oplus F(-H)\bigr) &
      \text{if }\varepsilon(Z)=1,
    \end{cases}
  \]
  and finally
  \[
    \Theta(Z)=\Psub\bigl(\Phalf(Z'\oplus F(G+2H))\oplus F(-H)\bigr).
  \]
  Since $\Padd,\Psub,\Phalf$ are total, there are tt-functionals
  $\widehat\varepsilon$ and $\widehat\Theta$ with
  $\varepsilon(Z)=\widehat\varepsilon(Z\oplus P)(0)$ and
  $\Theta(Z)\oplus P=\widehat\Theta(Z\oplus P)$ for all $Z\in\Cantor$.

  We claim that for every $m$,
  \begin{equation}\label{eq:bstep}
    \varepsilon\bigl(a(X_m)\bigr)=X(m)\qquad\text{and}\qquad
    \Theta\bigl(a(X_m)\bigr)=a(X_{m+1}).
  \end{equation}
  Let $Z=a(X_m)$. By \eqref{eq:btrans} with $T=X_m$ and $\kappa=\kappa_c$, the
  real tested in the definition of $\varepsilon(Z)$ is $a(X_m+\kappa_c)$. So
  by \eqref{eq:bread}, $\varepsilon(Z)=X_m(l)=X(m)$. Let
  $T'=0^{l+1}\cat S^{m+1}(X)$. If $X(m)=0$ then $T'=X_m$ and $Z'=Z=a(T')$. If
  $X(m)=1$ then $X_m=T'+Y$, so $T'=X_m-Y\in C_X$, and by \eqref{eq:btrans}
  with $T=X_m$ and $\kappa=-Y$ we get $Z'=a(X_m-Y)=a(T')$. In both cases
  $T'\in C_X\cap N_{0^{l+1}}$ and $T'/2=X_{m+1}\in C_X$, so by
  \eqref{eq:bhalf}, $\Theta(Z)=a(T'/2)=a(X_{m+1})$. This proves
  \eqref{eq:bstep}.

  \medskip
  \noindent\emph{Conclusion.} By \eqref{eq:bstep} and induction,
  $\widehat\Theta^{m}(a(X_0)\oplus P)=a(X_m)\oplus P$ and hence
  \[
    X(m)=\widehat\varepsilon\bigl(\widehat\Theta^{m}(a(X_0)\oplus P)\bigr)(0)
    \qquad\text{for all }m .
  \]
  The right-hand side is a tt-functional applied to $a(X_0)\oplus P$, so
  $X\ltt a(X_0)\oplus P$. It remains to bound $a(X_0)$ and $P$, using
  Lemma~\ref{lem:basicD}\eqref{bd:mj} and \eqref{eq:inD}. We have
  $a(X_0)=F(G+X_0)$ with $G+X_0\in D$, $G\in D$ and $G+X_0\ltt X\oplus G$; so
  $a(X_0)\ltt A\oplus F(G)$. Each of $H+\kappa_c$, $H-Y$, $-H$ belongs to $D$
  and is truth-table reducible to $H\in D$, and $G+2H\in D$ with
  $G+2H\ltt G\oplus H$; so $P\ltt F(G)\oplus F(H)$. Therefore
  $X\ltt A\oplus F(G)\oplus F(H)$.
\end{proof}

\subsection{Removing the parameters}\label{ss:remove}

\begin{lemma}\label{lem:remove}
  Let $U\subseteq\Cantor\times\Cantor$ be a nonempty open set and let
  $B,X\in\Cantor$. If $B\ltt X\oplus G\oplus H$ for comeager many $(G,H)\in U$,
  then $B\ltt X$.
\end{lemma}

\begin{proof}
  For each tt-functional $\Phi$ the set
  $C_\Phi=\{(G,H)\in U:B=\Phi(X\oplus G\oplus H)\}$ is closed in $U$. The union
  of these countably many sets is comeager in the nonempty Baire space $U$, so
  it is not meager, and hence some $C_\Phi$ is not nowhere dense. Being closed
  in $U$, this $C_\Phi$ contains a basic open set $N_\gamma\times N_\delta$.
  Then $B=\Phi(X\oplus\gamma\cat 0^{\omega}\oplus\delta\cat 0^{\omega})$, so
  $B\ltt X$.
\end{proof}

\begin{proposition}\label{prop:half}
  Assume \eqref{eq:data}. Then $\pi^{-1}(\mathbf{x})\le\mathbf{x}$ for every
  $\mathbf{x}\in\Dtt$; equivalently, $\mathbf{x}\le\pi(\mathbf{x})$ for every
  $\mathbf{x}\in\Dtt$.
\end{proposition}

\begin{proof}
  Let $X\in\Cantor$ and let $B\in\Cantor$ with $\dg(B)=\pi^{-1}(\dg(X))$. Let
  $U$ and $\mathcal{G}_X$ be as in Proposition~\ref{prop:coding}, and fix any
  $A$ with $\dg(A)=\pi(\dg(X))$. For every $(G,H)\in\mathcal{G}_X$ we have
  $G,H\in D$ and $X\ltt A\oplus F(G)\oplus F(H)$, that is,
  \[
    \dg(X)\le\pi(\dg(X))\vee\pi(\dg(G))\vee\pi(\dg(H))
    =\pi\bigl(\dg(X\oplus G\oplus H)\bigr).
  \]
  Applying the order isomorphism $\pi^{-1}$ gives
  $\dg(B)=\pi^{-1}(\dg(X))\le\dg(X\oplus G\oplus H)$. So
  $B\ltt X\oplus G\oplus H$ for comeager many $(G,H)\in U$, and
  Lemma~\ref{lem:remove} gives $B\ltt X$. This shows
  $\pi^{-1}(\mathbf{x})\le\mathbf{x}$ for all $\mathbf{x}$; applying this to
  $\pi(\mathbf{x})$ gives $\mathbf{x}\le\pi(\mathbf{x})$.
\end{proof}

Note that no function inducing $\pi^{-1}$ was needed for
Proposition~\ref{prop:half}: the inverse automorphism is applied to degrees
only.

\subsection{The graph of \texorpdfstring{$\pi$}{pi} is Borel}

To obtain the reverse inequality we will apply Proposition~\ref{prop:half} to
$\pi^{-1}$, and for this we need a continuous function inducing $\pi^{-1}$ on a
comeager set. The key is that $\pi$ is determined by $F$, in a sufficiently
definable way: every degree is a meet of two joins of degrees of elements of
$D$. (Below, the indices are chosen with $G_2=X-G_1$ and $G_4=X-G_3$ in mind,
so that $\dg(X)=(\dg(G_1)\vee\dg(G_2))\wedge(\dg(G_3)\vee\dg(G_4))$.)

\begin{lemma}[Meets of generic joins]\label{lem:meet}
  Let $Q$ be the set of triples $(X,G_1,G_3)\in(\Cantor)^{3}$ such that for all
  tt-functionals $\Phi$ and $\Psi$, either
  $\Phi(X\oplus G_1)\ne\Psi(X\oplus G_3)$, or there are strings
  $\gamma\subset G_1$ and $\delta\subset G_3$ such that
  \[
    \Phi(X\oplus G_1')=\Psi(X\oplus G_3')\qquad\text{for all }
    G_1'\in N_\gamma,\ G_3'\in N_\delta .
  \]
  Then:
  \begin{enumerate}
  \item\label{mt:borel} $Q$ is a Borel set.
  \item\label{mt:comeager} For every $X$, the section
    $Q_X=\{(G_1,G_3):(X,G_1,G_3)\in Q\}$ is comeager in
    $\Cantor\times\Cantor$.
  \item\label{mt:meet} If $(X,G_1,G_3)\in Q$, then $\dg(X)$ is the greatest
    lower bound of $\dg(X\oplus G_1)$ and $\dg(X\oplus G_3)$ in $\Dtt$.
  \end{enumerate}
\end{lemma}

\begin{proof}
  \eqref{mt:borel} Fix $\Phi,\Psi$. The set of $(X,G_1,G_3)$ with
  $\Phi(X\oplus G_1)\ne\Psi(X\oplus G_3)$ is open. For strings $\gamma,\delta$,
  the set
  \[
    L_{\gamma,\delta}=\bigl\{X:\Phi(X\oplus G_1')=\Psi(X\oplus G_3')
    \text{ for all }G_1'\in N_\gamma,\ G_3'\in N_\delta\bigr\}
  \]
  is closed, since its complement is the projection to the first coordinate of
  an open subset of $\Cantor\times N_\gamma\times N_\delta$, and projections
  are open maps. Now $Q$ is the intersection, over the countably many pairs
  $(\Phi,\Psi)$, of the union of an open set with the countably many closed
  sets $L_{\gamma,\delta}\times N_\gamma\times N_\delta$.

  \eqref{mt:comeager} Fix $X$, and for $\Phi,\Psi$ let
  $C=C_{\Phi,\Psi}=\{(G_1,G_3):\Phi(X\oplus G_1)=\Psi(X\oplus G_3)\}$, a closed
  set. The condition in the definition of $Q$ for $(\Phi,\Psi)$ says exactly
  that $(G_1,G_3)$ is not in $C$ or is in the interior of $C$. So the
  complement of $Q_X$ is the union over $(\Phi,\Psi)$ of the sets
  $C\setminus\operatorname{int}(C)$, which are closed and nowhere dense.

  \eqref{mt:meet} Clearly $\dg(X)$ is a lower bound. Let $B\ltt X\oplus G_1$
  and $B\ltt X\oplus G_3$, say $B=\Phi(X\oplus G_1)=\Psi(X\oplus G_3)$. Since
  $(X,G_1,G_3)\in Q$ there are $\gamma\subset G_1$ and $\delta\subset G_3$ as
  in the definition. Taking $G_1'=G_1$ and $G_3'=\delta\cat 0^{\omega}$ we get
  $B=\Psi(X\oplus\delta\cat 0^{\omega})$, so $B\ltt X$.
\end{proof}

\begin{lemma}\label{lem:graph}
  Assume \eqref{eq:data}. Then the set
  \[
    R_\pi=\bigl\{(X,A)\in\Cantor\times\Cantor:\dg(A)=\pi(\dg(X))\bigr\}
  \]
  is Borel.
\end{lemma}

\begin{proof}
  Let
  \[
    \mathcal{W}=\bigl\{(X,G_1,G_3)\in Q:\ G_1,\ X-G_1,\ G_3,\ X-G_3\in D\bigr\},
  \]
  where $Q$ is as in Lemma~\ref{lem:meet}. Then $\mathcal{W}$ is Borel, since
  $D$ is a $G_\delta$ set and the maps $(X,G)\mapsto X-G$ are continuous. For
  each $X$ the section $\mathcal{W}_X$ is comeager in $\Cantor\times\Cantor$,
  hence nonempty: it is the intersection of $Q_X$ with the four dense
  $G_\delta$ sets $D\times\Cantor$, $(X-D)\times\Cantor$, $\Cantor\times D$ and
  $\Cantor\times(X-D)$, where $X-D=\{X-G:G\in D\}$ is the image of $D$ under
  the homeomorphism $G\mapsto X-G$. On $\mathcal{W}$ define the continuous
  functions
  \[
    K_1(X,G_1,G_3)=F(G_1)\oplus F(X-G_1),\qquad
    K_3(X,G_1,G_3)=F(G_3)\oplus F(X-G_3).
  \]

  \medskip
  \noindent\emph{Claim.} If $(X,G_1,G_3)\in\mathcal{W}$ then $\pi(\dg(X))$ is
  the greatest lower bound of $\dg(K_1)$ and $\dg(K_3)$.

  \smallskip
  Indeed, $G_1\oplus(X-G_1)\ett X\oplus G_1$ by
  Lemma~\ref{lem:2adic}\eqref{2a:tt}, since $X=G_1+(X-G_1)$. As $F$ induces
  $\pi$ on $D$ and $\pi$ preserves joins,
  \begin{align*}
    \dg(K_1)&=\pi(\dg(G_1))\vee\pi(\dg(X-G_1))\\
    &=\pi\bigl(\dg(G_1)\vee\dg(X-G_1)\bigr)=\pi\bigl(\dg(X\oplus G_1)\bigr),
  \end{align*}
  and likewise $\dg(K_3)=\pi(\dg(X\oplus G_3))$. By
  Lemma~\ref{lem:meet}\eqref{mt:meet}, $\dg(X)$ is the greatest lower bound of
  $\dg(X\oplus G_1)$ and $\dg(X\oplus G_3)$, and an automorphism of a partial
  order maps the greatest lower bound of two elements to the greatest lower
  bound of their images. This proves the claim.

  \medskip
  Let $M$ be the set of triples $(A,K,K')\in(\Cantor)^{3}$ such that
  \[
    A\ltt K,\quad A\ltt K',\quad\text{and for all }B\in\Cantor,
    \text{ if }B\ltt K\text{ and }B\ltt K'\text{ then }B\ltt A .
  \]
  Thus $(A,K,K')\in M$ if and only if $\dg(A)$ is the greatest lower bound of
  $\dg(K)$ and $\dg(K')$. The relation $\ltt$ is a countable union of closed
  sets, hence Borel. The last clause in the definition of $M$ is equivalent to:
  for all tt-functionals $\Phi$ and $\Psi$, if $\Phi(K)=\Psi(K')$ then
  $\Phi(K)\ltt A$. This is a countable conjunction of Borel conditions, so $M$
  is Borel.

  By the claim, and since $\mathcal{W}_X$ is nonempty for every $X$, the
  following are equivalent for all $X$ and $A$:
  \begin{enumerate}
  \item[(a)] $(X,A)\in R_\pi$;
  \item[(b)] for all $(G_1,G_3)\in\mathcal{W}_X$, we have
    $\bigl(A,K_1(X,G_1,G_3),K_3(X,G_1,G_3)\bigr)\in M$;
  \item[(c)] for some $(G_1,G_3)\in\mathcal{W}_X$, we have
    $\bigl(A,K_1(X,G_1,G_3),K_3(X,G_1,G_3)\bigr)\in M$.
  \end{enumerate}
  Indeed, (a) implies (b) by the claim; (b) implies (c) because
  $\mathcal{W}_X\ne\emptyset$; and if (c) holds then $\dg(A)$ and
  $\pi(\dg(X))$ are both the greatest lower bound of $\dg(K_1)$ and
  $\dg(K_3)$ for the witnessing pair, so (a) holds. The set
  \[
    \bigl\{(X,A,G_1,G_3):(X,G_1,G_3)\in\mathcal{W}\text{ and }
    \bigl(A,K_1(X,G_1,G_3),K_3(X,G_1,G_3)\bigr)\in M\bigr\}
  \]
  is Borel, being the preimage of the Borel set $M$ under a continuous function
  defined on a Borel set. By (c), $R_\pi$ is its projection to the coordinates
  $(X,A)$, hence analytic. Similarly, by (b) the complement of $R_\pi$ is the
  projection of the Borel set obtained by replacing ``$\in M$'' by
  ``$\notin M$'', hence analytic. So $R_\pi$ is both analytic and coanalytic,
  and therefore Borel by Suslin's theorem \cite[14.11]{Kec}.
\end{proof}

\subsection{The inverse automorphism}\label{ss:inverse}

\begin{proposition}\label{prop:inverseD}
  Assume \eqref{eq:data}. Then there are a tt-functional $\Psi$ and a dense
  $G_\delta$ set $\widetilde D\subseteq\Cantor$ such that $\Psi\res\widetilde D$
  induces $\pi^{-1}$ on $\widetilde D$.
\end{proposition}

\begin{proof}
  For each tt-functional $\Phi$ let
  \[
    B_\Phi=\bigl\{Z\in\Cantor:\dg(\Phi(Z))=\pi^{-1}(\dg(Z))\bigr\}
    =\bigl\{Z\in\Cantor:(\Phi(Z),Z)\in R_\pi\bigr\}.
  \]
  By Lemma~\ref{lem:graph}, $B_\Phi$ is the preimage of a Borel set under
  a continuous map, hence Borel, and therefore has the Baire property. By Proposition~\ref{prop:half}, for every $Z$ the degree
  $\pi^{-1}(\dg(Z))$ is below $\dg(Z)$, so it contains a real of the form
  $\Phi(Z)$; thus $\Cantor=\bigcup_\Phi B_\Phi$. Since $\Cantor$ is not meager,
  some $B_\Phi$ is not meager. As it has the Baire property, it is comeager in
  some nonempty open set \cite[8.26]{Kec}; so there are
  $\sigma\in\strings$ and a dense $G_\delta$ subset $D_1$ of $N_\sigma$ with
  $D_1\subseteq B_\Phi$. Let
  \[
    \Psi(Z)=\Phi(\sigma\cat Z),\qquad
    \widetilde D=\{Z\in\Cantor:\sigma\cat Z\in D_1\}.
  \]
  Then $\Psi$ is a tt-functional, $\widetilde D$ is a dense $G_\delta$ subset
  of $\Cantor$, and for $Z\in\widetilde D$ we have
  $\dg(\Psi(Z))=\pi^{-1}(\dg(\sigma\cat Z))=\pi^{-1}(\dg(Z))$.
\end{proof}

\subsection{Proof of the main theorem}

We first record the version of the theorem in which a function is given for
the inverse automorphism as well. This is the form in which the coding lemma
is used in \cite{Mar}, and its proof needs only
Sections~\ref{ss:E}--\ref{ss:remove}.

\begin{proposition}[Two-sided version]\label{prop:twosided}
  Let $\pi$ be an automorphism of $\Dtt$. Suppose that $\pi$ is induced on a
  comeager set by a continuous function, and that $\pi^{-1}$ is induced on a
  comeager set by a continuous function. Then $\pi$ is the identity.
\end{proposition}

\begin{proof}
  A comeager set contains a dense $G_\delta$ set, and the restriction of a
  continuous function inducing an automorphism on a set induces it on every
  subset. So \eqref{eq:data} holds for $\pi$ with some $D$ and $F$, and it
  holds for $\pi^{-1}$ in place of $\pi$ with some $\widetilde D$ and
  $\widetilde F$. By Proposition~\ref{prop:half} applied to $\pi$,
  \[
    \mathbf{x}\le\pi(\mathbf{x})\qquad\text{for all }\mathbf{x}\in\Dtt ,
  \]
  and by Proposition~\ref{prop:half} applied to $\pi^{-1}$,
  \[
    \mathbf{x}\le\pi^{-1}(\mathbf{x})\qquad\text{for all }\mathbf{x}\in\Dtt .
  \]
  Applying $\pi$ to the last inequality gives $\pi(\mathbf{x})\le\mathbf{x}$.
  Hence $\pi(\mathbf{x})=\mathbf{x}$ for all $\mathbf{x}$.
\end{proof}

\begin{proof}[Proof of Theorem~\ref{thm:main}]
  Let $\pi$, $D$ and $F$ be as in the theorem. Replacing $D$ by a dense
  $G_\delta$ subset, we may assume that \eqref{eq:data} holds. By
  Proposition~\ref{prop:inverseD} there are a tt-functional $\Psi$ and a dense
  $G_\delta$ set $\widetilde D$ such that $\Psi\res\widetilde D$ induces
  $\pi^{-1}$ on $\widetilde D$. So the hypotheses of
  Proposition~\ref{prop:twosided} hold, and $\pi$ is the identity.
\end{proof}

\begin{remark}[On the hypothesis]\label{rem:onesided}
  It may be helpful to separate the two steps of the proof of
  Theorem~\ref{thm:main}, since they are of different character.

  The two-sided statement, Proposition~\ref{prop:twosided}, is the direct
  analogue of the argument in \cite{Mar}. Its proof consists of the coding
  lemma and the removal of parameters, and uses no descriptive set theory
  beyond the Baire category theorem. In particular it gives: \emph{if both
  $\pi$ and $\pi^{-1}$ are induced by Baire measurable functions, then $\pi$ is
  the identity.}

  That the hypothesis on $\pi^{-1}$ can be dropped, so that
  Theorem~\ref{thm:main} and Corollary~\ref{cor:bm} are correct as stated with
  a hypothesis on $\pi$ only, is the content of
  Proposition~\ref{prop:inverseD}: a continuous function inducing $\pi$ on a
  comeager set yields a tt-functional inducing $\pi^{-1}$ on a comeager set.
  Its proof uses, in addition, Proposition~\ref{prop:half} for $\pi$ and the
  fact that the sets $B_\Phi$ have the Baire property, which is where
  Lemmas~\ref{lem:meet} and~\ref{lem:graph} enter. For this purpose it is
  enough to know that $R_\pi$ is analytic, that is, the equivalence of (a)
  and (c) in the proof of Lemma~\ref{lem:graph}: then each $B_\Phi$ is
  analytic, as a continuous preimage of an analytic set, and analytic sets
  have the Baire property \cite[21.6]{Kec}. Alternatively one uses that
  $R_\pi$ is Borel, as we did. A third route avoids analytic sets altogether:
  with the notation of the proof of Lemma~\ref{lem:graph}, the set
  \[
    \bigl\{(Z,G_1,G_3):(\Phi(Z),G_1,G_3)\in\mathcal{W}\text{ and }
    \bigl(Z,K_1(\Phi(Z),G_1,G_3),K_3(\Phi(Z),G_1,G_3)\bigr)\in M\bigr\}
  \]
  is Borel, and by the equivalence of (a), (b) and (c) its section at $Z$ is
  the comeager set $\mathcal{W}_{\Phi(Z)}$ if $Z\in B_\Phi$ and is empty
  otherwise. So $B_\Phi$ is the set of $Z$ at which a Borel subset of
  $\Cantor\times(\Cantor\times\Cantor)$ has a nonmeager section, and such a
  set has the Baire property by the Kuratowski--Ulam theorem
  \cite[8.41]{Kec}. No function inducing $\pi^{-1}$ is assumed at any point:
  in Proposition~\ref{prop:half} the automorphism $\pi^{-1}$ is applied to
  degrees only, and in Lemma~\ref{lem:graph} only the fact that $\pi$
  preserves joins and meets is used.
\end{remark}

\begin{proof}[Proof of Corollary~\ref{cor:bm}]
  If $F\colon\Cantor\to\Cantor$ is Baire measurable and induces $\pi$, let $D$
  be a dense $G_\delta$ set such that $F\res D$ is continuous
  \cite[8.38]{Kec}. Then $F\res D$ induces $\pi$ on $D$, and
  Theorem~\ref{thm:main} applies.
\end{proof}

\begin{remark}
  In the continuous case we obtained $X\ltt F(X)$ outright
  (Proposition~\ref{prop:lower}), without passing through the parameters
  $F(G)$ and $F(H)$ and Lemma~\ref{lem:remove}. The reason is that there the
  relations \Radd, \Rsub, \Rhalf\ hold on a full open set of pairs, so that $G$
  and $H$ may be taken computable, and $F$ sends computable reals to computable
  reals. Likewise, in the continuous case the sets used in place of the $B_\Phi$ of
  Proposition~\ref{prop:inverseD} are closed (Proposition~\ref{prop:inverse}),
  so that Lemma~\ref{lem:graph} is not needed.
\end{remark}

\begin{remark}\label{rem:local}
  The proof uses $D$ only locally. If $D$ is merely comeager in some nonempty
  open set $N_\rho$ and $F\colon D\to\Cantor$ is continuous and induces $\pi$
  on $D$, then $Z\mapsto F(\rho\cat Z)$ is continuous on the comeager set
  $\{Z:\rho\cat Z\in D\}$ and induces $\pi$ on it, since
  $\rho\cat Z\ett Z$. So Theorem~\ref{thm:main} applies in this situation as
  well.
\end{remark}

\section{Determinacy}\label{sec:ad}

In this section we work in $\mathrm{ZF}+\mathrm{AD}$, where $\mathrm{AD}$ is
the Axiom of Determinacy: every two-player game of length $\omega$ with
integer moves is determined. Recall that $\mathrm{AD}$ implies that every set
of reals has the Baire property, and that it implies countable choice for sets
of reals, $\mathrm{AC}_\omega(\mathbb{R})$; see e.g.\ \cite{Mos}. Of course
$\mathrm{AD}$ contradicts the full axiom of choice.

The theory $\mathrm{ZF}+\mathrm{AD}$ does not obviously provide a way to
choose a representative of $\pi(\dg(X))$ for every $X$ simultaneously, so it is
not clear that an automorphism of $\Dtt$ is induced by a function defined on
all of $\Cantor$.
But Theorem~\ref{thm:main} only asks for a function on a comeager set, and this
is provided by the following standard consequence of $\mathrm{AD}$, whose
proof by an unfolded Banach--Mazur game we include for completeness.

\begin{lemma}[$\mathrm{ZF}+\mathrm{AD}$; comeager uniformization]\label{lem:cu}
  Let $R\subseteq\Cantor\times\Cantor$ be such that for every $X$ there is $A$
  with $(X,A)\in R$. Then there are a dense $G_\delta$ set $D\subseteq\Cantor$
  and a continuous function $F\colon D\to\Cantor$ such that $(X,F(X))\in R$ for
  all $X\in D$.
\end{lemma}

\begin{proof}
  Consider the following game. Players I and II move alternately; I plays
  nonempty strings $s_0,s_1,\dots\in\strings$, and II plays pairs
  $(t_0,a_0),(t_1,a_1),\dots$, where each $t_n\in\strings$ is nonempty and
  $a_n\in\{0,1\}$:
  \[
    \begin{array}{c|ccccc}
      \text{I} & s_0 & & s_1 & & \cdots\\ \hline
      \text{II} & & (t_0,a_0) & & (t_1,a_1) & \cdots
    \end{array}
  \]
  Such a play produces the reals
  $X=s_0\cat t_0\cat s_1\cat t_1\cat\cdots$ and $A=(a_0,a_1,\dots)$, and II
  wins if and only if $(X,A)\in R$. The moves can be coded by integers, so by
  $\mathrm{AD}$ one of the players has a winning strategy. For a finite
  sequence $p$ of moves let $x_p\in\strings$ be the concatenation of the
  strings played in $p$. We fix an enumeration of $\strings$ in order type
  $\omega$; ``least'' refers to it.

  \medskip
  \noindent\emph{Player I has no winning strategy.} Let $\varsigma$ be a
  strategy for I. For every position $p$ in which I has just moved according to
  $\varsigma$, and every $a\in\{0,1\}$, let
  \[
    O_{p,a}=(\Cantor\setminus N_{x_p})\ \cup\bigcup_{t\ne\emptyset}
    N_{x_p\cat t\cat\varsigma(p\cat(t,a))},
  \]
  where $\varsigma(p\cat(t,a))$ is the string with which $\varsigma$ answers
  the move $(t,a)$ of II in position $p$. This set is open, and it is dense:
  a basic open set $N_\nu$ with $\nu$ incompatible with $x_p$ is contained in
  it, and if $\nu$ is compatible with $x_p$ then $N_\nu$ contains $N_{x_p\cat
  t}$ for some nonempty $t$, hence it contains
  $N_{x_p\cat t\cat\varsigma(p\cat(t,a))}$. There are countably many pairs
  $(p,a)$, so by the Baire category theorem there is $X$ in all the sets
  $O_{p,a}$ with $X\supset s_0$, where $s_0$ is the first move of $\varsigma$.
  By hypothesis there is $A$ with $(X,A)\in R$. We define a play against
  $\varsigma$ by recursion, maintaining that the position $p_n$ after I's
  $n$th move $s_n$ satisfies $x_{p_n}\subset X$: this holds for
  $p_0=(s_0)$, and given $p_n$, since $X\in O_{p_n,A(n)}\cap N_{x_{p_n}}$ there
  is a nonempty $t$ with
  $x_{p_n}\cat t\cat\varsigma(p_n\cat(t,A(n)))\subset X$; let II play
  $(t,A(n))$ for the least such $t$. The resulting play is consistent with
  $\varsigma$ and produces $X$ and $A$, so II wins it. Hence $\varsigma$ is not
  a winning strategy.

  \medskip
  \noindent\emph{A winning strategy for II gives $F$.} By $\mathrm{AD}$, II has
  a winning strategy $\vartheta$. For every position $p$ in which II has just moved
  according to $\vartheta$ (including the empty position), let
  \[
    O_p=(\Cantor\setminus N_{x_p})\ \cup\bigcup_{s\ne\emptyset}
    N_{x_p\cat s\cat t(p,s)},
  \]
  where $(t(p,s),a(p,s))=\vartheta(p\cat s)$ is the answer of $\vartheta$ to the move $s$
  of I in position $p$. As before, each $O_p$ is open and dense. Let
  $D=\bigcap_p O_p$, a dense $G_\delta$ set. For $X\in D$ define a play by
  recursion, maintaining that the position $p_n$ after II's $n$th move satisfies
  $x_{p_n}\subset X$: $p_0$ is the empty position, and given $p_n$, since
  $X\in O_{p_n}\cap N_{x_{p_n}}$ there is a nonempty $s$ with
  $x_{p_n}\cat s\cat t(p_n,s)\subset X$; let I play the least such $s$ and let
  II answer by $\vartheta$. This play is consistent with $\vartheta$ and produces $X$
  together with some $A$, which we call $F(X)$. Since $\vartheta$ is winning,
  $(X,F(X))\in R$. Finally, $F$ is continuous on $D$: each position $p_n$,
  and hence $F(X)\res n$, is determined by a finite initial segment of $X$,
  because the least $s$ with $x_{p_n}\cat s\cat t(p_n,s)\subset X$ is
  determined by a finite initial segment of $X$ (to see that finitely many
  smaller candidates fail and that this one succeeds).
\end{proof}

\begin{proof}[Proof of Theorem~\ref{thm:ad}]
  Let $\pi$ be an automorphism of $\Dtt$ and let
  $R=\{(X,A):\dg(A)=\pi(\dg(X))\}$. Every section of $R$ is nonempty, so by
  Lemma~\ref{lem:cu} there are a dense $G_\delta$ set $D$ and a continuous
  $F\colon D\to\Cantor$ which induces $\pi$ on $D$. Thus the hypothesis of
  Theorem~\ref{thm:main} holds, and it remains to observe that the proof of
  Theorem~\ref{thm:main} given in Section~\ref{sec:general} goes through in
  $\mathrm{ZF}+\mathrm{AD}$. We go through the places where the axiom of
  choice might be used.
  \begin{itemize}
  \item The tt-functionals are indexed by a subset of $\omega$, so all
    ``countable'' unions and intersections over tt-functionals are over
    explicitly enumerated families, and each truth-table degree is countable.
  \item A dense $G_\delta$ subset of $\Cantor$ or of $\Cantor\times\Cantor$, or
    of a basic open subset of one of these spaces, will always be given
    together with a sequence of dense open sets whose intersection it is. For
    such sets, and for their relatively open subsets, the Baire category
    theorem is provable in $\mathrm{ZF}$, since these spaces have explicit
    countable bases. Lemma~\ref{lem:unif} is applied only to spaces of this
    kind (the general fact that Polish spaces are Baire is not needed).
    Similarly, a countable subset of such a space is meager, so that a set
    which is comeager in a nonempty open set is uncountable; this is used in
    Lemma~\ref{lem:split}.
  \item All dense $G_\delta$ and comeager sets occurring in
    Lemmas~\ref{lem:E} and~\ref{lem:split} and in
    Proposition~\ref{prop:coding} are obtained from the given sequence of dense
    open sets with intersection $D$ by explicit operations: preimages under
    homeomorphisms and open maps, and intersections indexed by tt-functionals,
    by strings, or by the group $C_X$, which is explicitly enumerated by words
    in its generators. In Proposition~\ref{prop:coding} the set $\mathcal{G}_X$
    is defined for each $X$ without any choice.
  \item In Lemma~\ref{lem:remove} the sets $C_\Phi$ are explicitly indexed. If
    all of them were nowhere dense, their union would be disjoint from a
    countable intersection of dense open subsets of $U$, by the Baire category
    theorem; so it could not contain a comeager subset of $U$. Only one
    sequence of dense open sets witnessing comeagerness is used.
  \item In Proposition~\ref{prop:half} the reals $A$ and $B$ are chosen for a
    single $X$ at a time.
  \item In Proposition~\ref{prop:inverseD} we need that the sets $B_\Phi$ have
    the Baire property, which under $\mathrm{AD}$ holds for all sets, so that
    Lemmas~\ref{lem:meet} and~\ref{lem:graph} are not needed; and that a
    countable union of meager sets is meager, which follows from
    $\mathrm{AC}_\omega(\mathbb{R})$. A nonmeager set $B_\Phi$ with the Baire
    property contains a dense $G_\delta$ subset $D_1$ of some $N_\sigma$; here
    a single sequence of dense open sets witnessing this is fixed, and it is
    the given sequence for $\widetilde D$ in the second application of
    Section~\ref{sec:general}, to $\pi^{-1}$, $\widetilde D$ and
    $\Psi\res\widetilde D$.
  \item In the proof of Theorem~\ref{thm:main} itself, a comeager set contains
    a dense $G_\delta$ set by definition, and again a single witnessing
    sequence is fixed.
  \end{itemize}
  Hence $\pi$ is the identity.
\end{proof}

\begin{remark}
  The proof of Theorem~\ref{thm:ad} uses $\mathrm{AD}$ only through three of its
  consequences: every set of reals has the Baire property, comeager
  uniformization (Lemma~\ref{lem:cu}), and $\mathrm{AC}_\omega(\mathbb{R})$.
  So the conclusion holds in any model of $\mathrm{ZF}$ in which these three
  statements hold.
\end{remark}

\begin{remark}
  Assume $\mathrm{ZFC}$ together with large cardinal hypotheses implying that
  $\mathrm{AD}$ holds in $L(\mathbb{R})$. Since $L(\mathbb{R})$ contains all
  reals, $\Dtt$ and its ordering are the same in $L(\mathbb{R})$ as in the
  universe. So by Theorem~\ref{thm:ad} applied in $L(\mathbb{R})$, no
  nontrivial automorphism of $\Dtt$ belongs to $L(\mathbb{R})$.
\end{remark}

\section{Other reducibilities}\label{sec:other}

The proofs above use only a few properties of the class of tt-functionals. In
this section we isolate these properties and observe that they hold for weak
truth-table reducibility, so that all the results of this note hold for the
wtt-degrees as well.

\subsection{Admissible classes of functionals}

By a \emph{partial functional} we mean a function $\Phi$ from a set
$\operatorname{dom}(\Phi)\subseteq\Cantor$ to $\Cantor$. A partial functional
$\Phi'$ \emph{extends} $\Phi$ if
$\operatorname{dom}(\Phi)\subseteq\operatorname{dom}(\Phi')$ and
$\Phi'(Y)=\Phi(Y)$ for all $Y\in\operatorname{dom}(\Phi)$. We use the following
operations on partial functionals $\Phi,\Psi$, where $j\in\omega$.
\begin{itemize}
\item \emph{Composition}: $(\Phi\circ\Psi)(Y)=\Phi(\Psi(Y))$, with domain
  $\{Y\in\operatorname{dom}(\Psi):\Psi(Y)\in\operatorname{dom}(\Phi)\}$.
\item \emph{Pairing}: $\langle\Phi,\Psi\rangle(Y)=\Phi(Y)\oplus\Psi(Y)$, with
  domain $\operatorname{dom}(\Phi)\cap\operatorname{dom}(\Psi)$.
\item \emph{Case distinction}: $[\Phi,\Psi]_j(Y)=\Phi(Y)$ if $Y(j)=0$ and
  $[\Phi,\Psi]_j(Y)=\Psi(Y)$ if $Y(j)=1$, with domain
  $\{Y\in\operatorname{dom}(\Phi):Y(j)=0\}\cup
  \{Y\in\operatorname{dom}(\Psi):Y(j)=1\}$.
\item \emph{Iteration}:
  $\operatorname{It}_j(\Phi,\Psi)(Y)=\bigl(\Psi(\Phi^{m}(Y))(j)\bigr)_{m\in\omega}$,
  where $\Phi^{m}$ is the $m$-fold composition of $\Phi$ with itself; the
  domain consists of all $Y$ such that for every $m$,
  $Y\in\operatorname{dom}(\Phi^{m})$ and
  $\Phi^{m}(Y)\in\operatorname{dom}(\Psi)$.
\end{itemize}

\begin{definition}\label{def:adm}
  An \emph{admissible class} is a countable set $\mathfrak{F}$ of partial
  functionals with the following properties.
  \begin{itemize}
  \item[\textup{(A0)}] For every $\Phi\in\mathfrak{F}$, the set
    $\operatorname{dom}(\Phi)$ is closed, and there is an oracle Turing machine
    which, for every $Y\in\operatorname{dom}(\Phi)$, halts on every input when
    run with oracle $Y$ and computes $\Phi(Y)$.
  \item[\textup{(A1)}] Every tt-functional belongs to $\mathfrak{F}$.
  \item[\textup{(A2)}] For all $\Phi,\Psi\in\mathfrak{F}$ and all $j\in\omega$,
    each of the partial functionals $\Phi\circ\Psi$, $\langle\Phi,\Psi\rangle$,
    $[\Phi,\Psi]_j$ and $\operatorname{It}_j(\Phi,\Psi)$ is extended by some
    element of $\mathfrak{F}$.
  \end{itemize}
  For an admissible class $\mathfrak{F}$ and $A,B\in\Cantor$ we write
  $A\le_{\mathfrak{F}}B$ if there is $\Phi\in\mathfrak{F}$ with
  $B\in\operatorname{dom}(\Phi)$ and $\Phi(B)=A$. We write
  $\equiv_{\mathfrak{F}}$ for the induced equivalence relation and
  $\mathcal{D}_{\mathfrak{F}}$ for the partial order of its equivalence
  classes, the \emph{$\mathfrak{F}$-degrees}.
\end{definition}

The class of all tt-functionals is admissible, and for it
$\le_{\mathfrak{F}}$ is $\ltt$. The following lemma collects what replaces, for
a general admissible class, the ``evident facts'' about tt-functionals listed
in Section~\ref{sec:prelim}.

\begin{lemma}\label{lem:adm}
  Let $\mathfrak{F}$ be an admissible class.
  \begin{enumerate}
  \item\label{adm:order} The relation $\le_{\mathfrak{F}}$ is reflexive and
    transitive, and for all $A,B$: $A\ltt B$ implies $A\le_{\mathfrak{F}}B$,
    and $A\le_{\mathfrak{F}}B$ implies $A\le_{\mathrm{T}}B$. Every
    $\mathfrak{F}$-degree is countable. The degree of $A\oplus B$ is the least
    upper bound of the degrees of $A$ and $B$ in $\mathcal{D}_{\mathfrak{F}}$,
    and the computable reals form the least element of
    $\mathcal{D}_{\mathfrak{F}}$.
  \item\label{adm:closed} Every $\Phi\in\mathfrak{F}$ is continuous on
    $\operatorname{dom}(\Phi)$, and its graph, the set of pairs $(Y,A)$ with
    $Y\in\operatorname{dom}(\Phi)$ and $\Phi(Y)=A$, is a closed subset of
    $\Cantor\times\Cantor$. Consequently, if $\mathcal{S}$ is a
    topological space and $\varphi,\psi\colon\mathcal{S}\to\Cantor$ are
    continuous, then the set of all $x\in\mathcal{S}$ such that
    $\psi(x)\in\operatorname{dom}(\Phi)$ and $\Phi(\psi(x))=\varphi(x)$ is
    closed in $\mathcal{S}$.
  \item\label{adm:total} If $\Phi\in\mathfrak{F}$, $\sigma\in\strings$ and
    $N_\sigma\subseteq\operatorname{dom}(\Phi)$, then $Z\mapsto\Phi(\sigma\cat Z)$
    is a tt-functional, and $\Phi$ agrees on $N_\sigma$ with a tt-functional.
  \end{enumerate}
\end{lemma}

\begin{proof}
  \eqref{adm:order} The identity is a tt-functional, so $\le_{\mathfrak{F}}$ is
  reflexive by (A1). If $A=\Phi(B)$ and $B=\Psi(C)$ with
  $\Phi,\Psi\in\mathfrak{F}$, then $C\in\operatorname{dom}(\Phi\circ\Psi)$ and
  $(\Phi\circ\Psi)(C)=A$, so $A\le_{\mathfrak{F}}C$ by (A2). The implication
  from $\ltt$ is (A1) and the implication to $\le_{\mathrm{T}}$ is (A0). Degrees
  are countable because $\mathfrak{F}$ is. The maps $A\oplus B\mapsto A$ and
  $A\oplus B\mapsto B$ are tt-functionals, so the degree of $A\oplus B$ is an upper
  bound; and if $A=\Phi(C)$ and $B=\Psi(C)$ then
  $A\oplus B=\langle\Phi,\Psi\rangle(C)$, so it is the least upper bound by
  (A2). A computable real is the value of a constant tt-functional, so it is
  $\le_{\mathfrak{F}}$ every real; and a real which is $\le_{\mathfrak{F}}$ a
  computable real is Turing reducible to it, hence computable.

  \eqref{adm:closed} By (A0), for $Y\in\operatorname{dom}(\Phi)$ each bit of
  $\Phi(Y)$ is computed by a halting oracle computation, which uses only
  finitely much of $Y$; so $\Phi$ is continuous on $\operatorname{dom}(\Phi)$.
  If $(Y_k,A_k)$ is a sequence in the graph converging to $(Y,A)$, then
  $Y\in\operatorname{dom}(\Phi)$ because the domain is closed, and
  $A_k=\Phi(Y_k)\to\Phi(Y)$ by continuity, so $A=\Phi(Y)$. Hence the graph is
  closed. The last statement follows, since the set in question is the
  preimage of the graph under the continuous map $x\mapsto(\psi(x),\varphi(x))$.

  \eqref{adm:total} By (A0) there is an oracle machine which halts on every
  input with every oracle in $N_\sigma$ and computes $\Phi$ there. Running it
  with oracle $\sigma\cat Z$ gives a machine which halts on every input with
  every oracle $Z$; so $Z\mapsto\Phi(\sigma\cat Z)$ is a tt-functional. The
  function which maps $Y\in N_\sigma$ to $\Phi(Y)$ and every other $Y$ to
  $0^{\omega}$ is then also a tt-functional.
\end{proof}

The next lemma does not involve an admissible class. It will show that the
left inverses in Theorem~\ref{thm:cts} can always be taken total.

\begin{lemma}\label{lem:inv}
  Let $\Lambda$ be a tt-functional and let $\tau\in\strings$ be such that
  $\Lambda$ is injective on $N_\tau$. Then there is a tt-functional $\Gamma$
  such that $\Gamma(\Lambda(Z))=Z$ for all $Z\in N_\tau$.
\end{lemma}

\begin{proof}
  Let $\mu$ be a computable nondecreasing function such that $\Lambda(Z)\res k$
  depends only on $Z\res\mu(k)$, for all $k$; such a function exists because
  $\Lambda$ is total. For $n,k\in\omega$ let $L(n,k)=\max(\mu(k),n,|\tau|)$, and
  for a string $\zeta$ of length at least $\mu(k)$ let $\Lambda(\zeta)\res k$
  denote the common value of $\Lambda(Z)\res k$ for $Z\supset\zeta$. Consider
  the decidable property
  \begin{quote}
    $P(n,k)$: for all strings $\zeta,\zeta'\supseteq\tau$ of length $L(n,k)$,
    if $\Lambda(\zeta)\res k=\Lambda(\zeta')\res k$ then
    $\zeta\res n=\zeta'\res n$.
  \end{quote}
  Equivalently, $P(n,k)$ says that for all $Z,Z'\in N_\tau$, if
  $\Lambda(Z)\res k=\Lambda(Z')\res k$ then $Z\res n=Z'\res n$. For every $n$
  there is $k$ with $P(n,k)$: otherwise, for every $k$ there are
  $Z_k,Z_k'\in N_\tau$ with $\Lambda(Z_k)\res k=\Lambda(Z_k')\res k$ and
  $Z_k\res n\ne Z_k'\res n$; by compactness we may pass to a subsequence along
  which $Z_k\to Z$ and $Z_k'\to Z'$, and then $Z,Z'\in N_\tau$,
  $Z\res n\ne Z'\res n$ and, by continuity, $\Lambda(Z)=\Lambda(Z')$,
  contradicting injectivity. Let $k(n)$ be the least $k$ with $P(n,k)$; this is
  a computable function.

  Define $\Gamma(A)(n)$ as follows. Let $k=k(n+1)$ and $L=L(n+1,k)$. If there is
  a string $\zeta\supseteq\tau$ of length $L$ with
  $\Lambda(\zeta)\res k=A\res k$, let $\Gamma(A)(n)=\zeta(n)$ for the least
  such $\zeta$; otherwise let $\Gamma(A)(n)=0$. Then $\Gamma$ is a
  tt-functional. If $A=\Lambda(Z)$ with $Z\in N_\tau$, then $\zeta=Z\res L$ is
  as required, and by $P(n+1,k)$ every $\zeta$ as required satisfies
  $\zeta\res(n+1)=Z\res(n+1)$; so $\Gamma(A)(n)=Z(n)$.
\end{proof}

\subsection{The results for admissible classes}

Let $\mathfrak{F}$ be an admissible class. Automorphisms of
$\mathcal{D}_{\mathfrak{F}}$, and functions inducing them (on a set $D$), are
defined as for $\Dtt$.

\begin{theorem}\label{thm:adm}
  Let $\mathfrak{F}$ be an admissible class. Then the following hold.
  \begin{enumerate}
  \item\label{at:main} If an automorphism $\pi$ of $\mathcal{D}_{\mathfrak{F}}$
    is induced on a comeager set by a continuous function, then $\pi$ is the
    identity. In particular this holds if $\pi$ is induced by a Baire
    measurable function.
  \item\label{at:cts} If an automorphism $\pi$ of
    $\mathcal{D}_{\mathfrak{F}}$ is induced by a continuous function
    $F\colon\Cantor\to\Cantor$, then every nonempty open subset of $\Cantor$
    contains a basic open set $N_\tau$ such that there are tt-functionals
    $\Lambda$ and $\Gamma$ with $F(Z)=\Lambda(Z)$ and $\Gamma(F(Z))=Z$ for all
    $Z\in N_\tau$. (That is, Theorem~\ref{thm:cts} holds for
    $\mathcal{D}_{\mathfrak{F}}$ as stated.)
  \item\label{at:perm} If $h$ is a permutation of $\omega$, then
    $X\mapsto X\circ h$ induces an automorphism of
    $\mathcal{D}_{\mathfrak{F}}$ if and only if $h$ is computable.
  \item\label{at:ad} In $\mathrm{ZF}+\mathrm{AD}$, every automorphism of
    $\mathcal{D}_{\mathfrak{F}}$ is the identity.
  \end{enumerate}
\end{theorem}

\begin{proof}
  The proofs of Sections~\ref{sec:prelim}--\ref{sec:ad} go through, with the
  changes that we now list. Throughout, ``tt-functional'' is replaced by
  ``element of $\mathfrak{F}$'', and $\ltt$, $\ett$, $\dg$, $\Dtt$ by the
  corresponding notions for $\mathfrak{F}$, \emph{except} where a functional is
  asserted to be a tt-functional because it is one of the specific total
  computable maps of Lemma~\ref{lem:2adic} or a map such as
  $Z\mapsto\sigma\cat Z$, $Z\mapsto Z\oplus P$ with $P$ computable, or a
  projection; these are tt-functionals, and belong to $\mathfrak{F}$ by (A1).
  For $\Phi\in\mathfrak{F}$, an equation ``$\Phi(Y)=A$'' is read as
  ``$Y\in\operatorname{dom}(\Phi)$ and $\Phi(Y)=A$'', and ``$\Phi(Y)\ne A$'' as
  its negation. In particular the relations \Radd, \Rsub, \Rhalf\ include the
  assertion that their right-hand sides are defined.

  \medskip
  \noindent\emph{Section~\ref{sec:prelim}.} The evident facts about
  tt-functionals are replaced by Lemma~\ref{lem:adm} and (A2).
  Lemmas~\ref{lem:basic} and~\ref{lem:basicD} use only
  Lemma~\ref{lem:adm}\eqref{adm:order}. In Lemma~\ref{lem:unif} the sets
  $C_\Phi$ are closed by Lemma~\ref{lem:adm}\eqref{adm:closed}. The reductions
  \eqref{eq:args} are truth-table reductions, hence
  $\le_{\mathfrak{F}}$-reductions. Lemma~\ref{lem:trans} is unchanged.

  \medskip
  \noindent\emph{Section~\ref{sec:cts}.} Steps~1 and~2 of the proof of
  Proposition~\ref{prop:lower} are unchanged. In Step~3, the functional
  $\Theta$ is obtained from $\Padd$, $\Psub$, $\Phalf$, the identity, and
  constant functionals with the computable values $P_1,P_2,P_3$ by
  compositions, pairings and one case distinction; by (A2) and (A1) it is
  extended by an element of $\mathfrak{F}$, which we again call $\Theta$. (If
  $\Phi'$ extends $\Phi$ and $\Psi'$ extends $\Psi$, then each of
  $\Phi'\circ\Psi'$, $\langle\Phi',\Psi'\rangle$, $[\Phi',\Psi']_j$,
  $\operatorname{It}_j(\Phi',\Psi')$ extends the corresponding functional for
  $\Phi$ and $\Psi$; so extensions may be taken at each stage of such a
  construction.) The verification of \eqref{eq:shift} is unchanged; it shows in
  particular that $F(\tau\cat W)\in\operatorname{dom}(\Theta)$ for all $W$. In
  Step~4, $\Gamma$ is an element of $\mathfrak{F}$ extending the composition of
  a tt-functional (which adds $b$ to every bit) with
  $\operatorname{It}_j(\Theta,\mathrm{id})$; the computation there shows that
  $F(\tau\cat X)\in\operatorname{dom}(\Gamma)$ and $\Gamma(F(\tau\cat X))=X$.
  So Proposition~\ref{prop:lower} holds, with $\Gamma\in\mathfrak{F}$, and
  Remark~\ref{rem:dichotomy} is unchanged.

  In the proof of Proposition~\ref{prop:inverse}, the set
  $C_{\Phi,\Lambda,\Lambda'}$ is the set of $Z\in\operatorname{dom}(\Phi)$ such
  that $(F(\Phi(Z)),Z)$ is in the graph of $\Lambda$ and $(Z,F(\Phi(Z)))$ is
  in the graph of $\Lambda'$. It is closed, because
  $\operatorname{dom}(\Phi)$ is closed, $F\circ\Phi$ is continuous on it, and
  the graphs are closed (Lemma~\ref{lem:adm}\eqref{adm:closed}). If
  $N_\sigma\subseteq C_{\Phi,\Lambda,\Lambda'}$ then
  $N_\sigma\subseteq\operatorname{dom}(\Phi)$, so
  $\widetilde F(Z)=\Phi(\sigma\cat Z)$ is a tt-functional by
  Lemma~\ref{lem:adm}\eqref{adm:total}. Thus $\pi^{-1}$ is induced by a
  tt-functional, and the proof of Theorem~\ref{thm:cts} goes through: in the
  ``moreover'' part, $\Lambda\in\mathfrak{F}$ has
  $N_{\rho'}\subseteq\operatorname{dom}(\Lambda)$, so by
  Lemma~\ref{lem:adm}\eqref{adm:total} it agrees on $N_{\rho'}$ with a
  tt-functional, which we again call $\Lambda$; and some element of
  $\mathfrak{F}$ extends $A\mapsto\tau\cat\Gamma_0(A)$ and so is a left inverse
  of $F$ on $N_\tau$. In particular $F$, and hence $\Lambda$, is injective on
  $N_\tau$, and Lemma~\ref{lem:inv} gives a tt-functional $\Gamma$ with
  $\Gamma(F(Z))=\Gamma(\Lambda(Z))=Z$ for all $Z\in N_\tau$. This proves
  \eqref{at:cts}, and that $\pi$ is the identity in this case; the remark
  following the proof of Theorem~\ref{thm:cts} also remains true as stated.
  The proof of Corollary~\ref{cor:perm} is then unchanged, since there
  $\Lambda$ is a tt-functional; this proves \eqref{at:perm}.

  \medskip
  \noindent\emph{Section~\ref{sec:general}.} Lemmas~\ref{lem:E}
  and~\ref{lem:split} are unchanged. In the proof of
  Proposition~\ref{prop:coding}, the functionals $\widehat\varepsilon$ and
  $\widehat\Theta$ are obtained from $\Padd,\Psub,\Phalf$ and tt-functionals by
  compositions, pairings and a case distinction, as follows. Write $V$ for
  $Z\oplus P$. For each $c<2^{d-1}$ the real
  $\Xi_c(V)=\Psub(\Padd(Z\oplus F(H+\kappa_c))\oplus F(-H))$ is obtained from
  $V$ by compositions and pairings, since $Z$, $F(H+\kappa_c)$ and $F(-H)$ are
  obtained from $V$ by tt-functionals. Then $\widehat\varepsilon(V)$ is
  obtained by applying a tt-functional, which inspects the $j$th bit of each
  component, to the join of the $\Xi_c(V)$; and $\widehat\Theta(V)$ is obtained
  by applying to $\widehat\varepsilon(V)\oplus V$ a case distinction, on the
  bit of this input which is $\widehat\varepsilon(V)(0)$, between two
  functionals built by compositions and pairings. So $\widehat\varepsilon$ and $\widehat\Theta$ are extended by
  elements of $\mathfrak{F}$, and the proof of \eqref{eq:bstep} shows that
  $a(X_m)\oplus P$ lies in their domains. The final decoding functional is an
  iteration, covered by (A2). In Lemma~\ref{lem:remove} the sets $C_\Phi$ are
  closed by Lemma~\ref{lem:adm}\eqref{adm:closed}; if
  $N_\gamma\times N_\delta\subseteq C_\Phi$ then in particular
  $X\oplus\gamma\cat 0^{\omega}\oplus\delta\cat 0^{\omega}\in
  \operatorname{dom}(\Phi)$, and $B\le_{\mathfrak{F}}X$ since
  $X\mapsto X\oplus\gamma\cat 0^{\omega}\oplus\delta\cat 0^{\omega}$ is a
  tt-functional. Proposition~\ref{prop:half} is unchanged.

  In Lemma~\ref{lem:meet}, with the above reading of equations, the set of
  $(X,G_1,G_3)$ with $\Phi(X\oplus G_1)=\Psi(X\oplus G_3)$ is closed by
  Lemma~\ref{lem:adm}\eqref{adm:closed} (it is the set of triples such that
  some $A$ has $(X\oplus G_1,A)$ in the graph of $\Phi$ and $(X\oplus G_3,A)$
  in the graph of $\Psi$, a projection of a closed subset of a product with
  the compact space $\Cantor$). So its complement is open, each
  $L_{\gamma,\delta}$ is closed, each $C_{\Phi,\Psi}$ is closed, and the proof
  is unchanged; in part \eqref{mt:meet}, $\Psi(X\oplus\delta\cat 0^{\omega})$
  is defined and equal to $B$. In Lemma~\ref{lem:graph}, the relation
  $\le_{\mathfrak{F}}$ is a countable union of closed sets by
  Lemma~\ref{lem:adm}\eqref{adm:closed}. The set $M$ is Borel as before: its
  last clause says that for all $\Phi,\Psi\in\mathfrak{F}$, if
  $\Phi(K)=\Psi(K')$ then $\Phi(K)\le_{\mathfrak{F}}A$, and for fixed $\Phi$,
  $\Psi$ the set of $(A,K,K')$ with $\Phi(K)=\Psi(K')$ is closed, while the set
  of $(A,K)$ with $K\in\operatorname{dom}(\Phi)$ and
  $\Phi(K)\le_{\mathfrak{F}}A$ is a countable union of closed sets. In
  Proposition~\ref{prop:inverseD} we let
  $B_\Phi=\{Z\in\operatorname{dom}(\Phi):(\Phi(Z),Z)\in R_\pi\}$, a Borel set.
  The set $D_1\subseteq B_\Phi$ is dense in $N_\sigma$ and contained in the
  closed set $\operatorname{dom}(\Phi)$, so
  $N_\sigma\subseteq\operatorname{dom}(\Phi)$ and $\Psi(Z)=\Phi(\sigma\cat Z)$
  is a tt-functional by Lemma~\ref{lem:adm}\eqref{adm:total}. The rest of
  Section~\ref{sec:general} is unchanged. This proves \eqref{at:main}.

  \medskip
  \noindent\emph{Section~\ref{sec:ad}.} Fix an enumeration of $\mathfrak{F}$.
  Lemma~\ref{lem:cu} does not involve the reducibility, and the proof of
  Theorem~\ref{thm:ad}, including the discussion of the use of choice, applies
  with the changes already described. This proves \eqref{at:ad}.
\end{proof}

\subsection{Weak truth-table reducibility}

Recall that $A$ is \emph{weak truth-table reducible} to $B$,
$A\le_{\mathrm{wtt}}B$, if there are an oracle Turing machine $\Phi_e$ and a
computable function $u\colon\omega\to\omega$ such that $\Phi_e^{B}=A$ and for
every $n$ the computation $\Phi_e^{B}(n)$ queries the oracle only at numbers
less than $u(n)$. For every $e$ and every total computable function $u$, let
$\Phi_{e,u}$ be the partial functional whose domain is the set of all
$Y\in\Cantor$ such that for every $n$,
\begin{quote}
  the computation $\Phi_e^{Y}(n)$ halts with output $0$ or $1$ and queries the
  oracle only at numbers less than $u(n)$,
\end{quote}
and let $\Phi_{e,u}(Y)=\Phi_e^{Y}$ on this domain. Let
$\mathfrak{F}_{\mathrm{wtt}}$ be the set of all these functionals. By
definition, $A\le_{\mathrm{wtt}}B$ if and only if
$A\le_{\mathfrak{F}_{\mathrm{wtt}}}B$.

\begin{proposition}\label{prop:wtt}
  $\mathfrak{F}_{\mathrm{wtt}}$ is an admissible class.
\end{proposition}

\begin{proof}
  The class is countable. Note first that if $u\le u'$ pointwise then
  $\Phi_{e,u'}$ extends $\Phi_{e,u}$. Since the four operations in (A2) are
  monotone with respect to extension (if $\Phi'$ extends $\Phi$ and $\Psi'$
  extends $\Psi$ then $\Phi'\circ\Psi'$ extends $\Phi\circ\Psi$, and so on),
  in verifying (A2) we may replace each bound $u$ by the nondecreasing bound
  $n\mapsto\max_{k\le n}u(k)$, that is, we may assume that all bounds are
  nondecreasing.

  (A0) For fixed $n$, whether $\Phi_e^{Y}(n)$ halts with output $0$ or $1$
  while querying only numbers less than $u(n)$ depends only on $Y\res u(n)$. So the set of such
  $Y$ is clopen, and $\operatorname{dom}(\Phi_{e,u})$, the intersection of
  these sets over all $n$, is closed. The machine $\Phi_e$ computes
  $\Phi_{e,u}$ on this domain.

  (A1) Let $\Phi$ be a tt-functional, computed by the machine $\Phi_e$, which
  halts on every input with every oracle. For each $n$, by compactness of
  $\Cantor$ there is $k$ such that for every string $\sigma$ of length $k$ the
  computation $\Phi_e^{\sigma}(n)$ halts having queried only numbers less than
  $k$; let $u(n)$ be the least such $k$, which can be found effectively. Then
  $u$ is computable and $\Phi=\Phi_{e,u}$.

  (A2) Let $\Phi=\Phi_{e,u}$ and $\Psi=\Phi_{e',u'}$ with $u,u'$ nondecreasing.

  \emph{Composition.} Consider the machine which, with oracle $Y$ and input
  $n$, first computes $\Phi_{e'}^{Y}(k)$ for all $k<u(n)$, obtaining a string
  $\rho$ of length $u(n)$, and then runs $\Phi_e^{\rho}(n)$, diverging if a
  number $\ge u(n)$ is queried. If $Y\in\operatorname{dom}(\Phi\circ\Psi)$
  then this halts with output $\Phi(\Psi(Y))(n)$, and it queries $Y$ only at
  numbers less than $u'(u(n))$.
  So the functional in $\mathfrak{F}_{\mathrm{wtt}}$ given by this machine and
  the bound $n\mapsto u'(u(n))$ extends $\Phi\circ\Psi$.

  \emph{Pairing.} The machine which computes $\Phi_e^{Y}(n)$ on input $2n$ and
  $\Phi_{e'}^{Y}(n)$ on input $2n+1$, with the bound $2n\mapsto u(n)$,
  $2n+1\mapsto u'(n)$, gives $\langle\Phi,\Psi\rangle$.

  \emph{Case distinction.} The machine which on input $n$ queries $Y(j)$ and
  then runs $\Phi_e^{Y}(n)$ or $\Phi_{e'}^{Y}(n)$ accordingly, with the bound
  $n\mapsto\max(u(n),u'(n),j+1)$, gives a functional extending
  $[\Phi,\Psi]_j$.

  \emph{Iteration.} Define $v(m,\ell)$ by $v(0,\ell)=\ell$ and
  $v(m+1,\ell)=v(m,u(\ell))$. By induction on $m$, if
  $Y\in\operatorname{dom}(\Phi^{m})$ then the first $\ell$ bits of
  $\Phi^{m}(Y)$ can be computed, uniformly in $m$ and $\ell$, by iterating the
  machine $\Phi_e$, querying $Y$ only at numbers less than $v(m,\ell)$: for
  $m+1$, compute the first $u(\ell)$ bits of $\Phi^{m}(Y)$ and run $\Phi_e$ on
  the resulting string for the inputs less than $\ell$. Now consider the
  machine which, with oracle $Y$ and input $m$, computes in this way the first
  $u'(j)$ bits of $\Phi^{m}(Y)$ and then runs $\Phi_{e'}$ on the resulting
  string with input $j$. For $Y\in\operatorname{dom}(\operatorname{It}_j(\Phi,
  \Psi))$ it halts with output $\Psi(\Phi^{m}(Y))(j)$, and it queries $Y$ only
  at numbers less than $v(m,u'(j))$, a computable function of $m$. So it gives
  an element of $\mathfrak{F}_{\mathrm{wtt}}$ extending
  $\operatorname{It}_j(\Phi,\Psi)$.
\end{proof}

\begin{corollary}\label{cor:wtt}
  All four statements of Theorem~\ref{thm:adm} hold for the wtt-degrees. In
  particular, an automorphism of the wtt-degrees which is induced by a Baire
  measurable function is the identity, and in $\mathrm{ZF}+\mathrm{AD}$ the
  wtt-degrees are rigid.
\end{corollary}

\subsection{Limits of the method}

\begin{remark}\label{rem:limits}
  (1) \emph{Turing reducibility.} The class of all Turing functionals,
  each restricted to the set of oracles on which it is total, satisfies (A1)
  and (A2), but not (A0): the domains are $G_\delta$ sets which need not be
  closed, and the graphs are $G_\delta$ rather than closed. In fact no
  admissible class induces Turing reducibility. If $\Phi$ belongs to an
  admissible class, then by (A0) and compactness of $\operatorname{dom}(\Phi)$
  the number of oracle bits used to compute $\Phi(Y)(n)$ is bounded, for
  $Y\in\operatorname{dom}(\Phi)$, by some $g_\Phi(n)$; so an admissible
  reducibility is always of ``weak truth-table type'', with use bounds from a
  countable family of functions which need not be computable. By a standard
  diagonalization, for every countable family of functions there are reals
  $A\le_{\mathrm{T}}B$ such that no Turing reduction of $A$ to $B$ has use
  bounded by a member of the family.

  So Theorem~\ref{thm:adm} does not apply to Turing reducibility, and in
  particular the elementary argument of Section~\ref{sec:cts} is not available
  for it. We expect that the arguments of Section~\ref{sec:general}, which take
  place on comeager sets throughout, can be adapted, with ``comeager in an
  open set'' in place of ``open'' in Lemma~\ref{lem:unif}; this would give a
  category version of \cite[Lemma~0.2]{Mar}. We have not carried this out,
  since the Turing degrees are rigid by \cite{OAI}.

  (2) \emph{Reducibilities below $\ltt$.} For bounded truth-table and many-one
  reducibility the argument does not start: the $n$th bit of $G+H$ depends on
  $n+1$ bits of each of $G$ and $H$, so the reductions \eqref{eq:args} are not
  given by bounded truth-table reductions uniformly on any open set of pairs.
  For bounded truth-table reducibility closure under iteration fails as well
  (iterating $A\mapsto(A(n)+A(n+1)\bmod 2)_{n\in\omega}$ produces bits that
  depend on unboundedly many bits of $A$), and for many-one reducibility
  case distinction fails. This is as it should be, since the many-one degrees
  have many automorphisms.

  (3) \emph{Reducibilities above $\le_{\mathrm{T}}$.} By
  Lemma~\ref{lem:adm}\eqref{adm:order}, a reducibility that is not contained in
  $\le_{\mathrm{T}}$ is not given by an admissible class; for arithmetic
  reducibility, where $A$ is reducible to $B$ if $A$ is computable from some
  finite jump of $B$, already (A0) fails, since the reductions are not
  continuous. There is also a second obstacle, which would remain after an
  adaptation as in (1): closure under iteration fails. The $m$-fold iterate of
  a reduction which uses a fixed positive number of jumps uses a number of
  jumps that grows with $m$, so the decoding procedure of
  Proposition~\ref{prop:coding} is not an arithmetic reduction.
\end{remark}

\section{Concluding remarks}

\subsection*{Comparison with the Turing case}
For Turing reducibility the analogue of the set $C_\Phi$ in
Lemma~\ref{lem:unif} is the set of $x$ such that a partial functional converges
to the right value on $\psi(x)$, which is $G_\delta$ rather than closed; this
is why the argument in \cite{Mar} needs sufficiently generic pairs $(G,H)$ even
when $F$ is continuous, and why it is natural there to phrase things in terms
of forcing. For truth-table reducibility the sets are closed, and this is what
makes Theorem~\ref{thm:cts} elementary. It is also what makes
Propositions~\ref{prop:inverse} and~\ref{prop:inverseD} possible: once
$\pi^{-1}(\mathbf{z})\le\mathbf{z}$ is known for all $\mathbf{z}$, a single
tt-functional computes a representative of $\pi^{-1}(\dg(Z))$ from $Z$ for all
$Z$ in a large set, and so the inverse automorphism is again induced by a
continuous function on a comeager set.

\subsection*{The representation problem}
By Corollary~\ref{cor:bm}, a nontrivial automorphism of $\Dtt$, if there is
one, cannot be induced by a Baire measurable function; and by
Theorem~\ref{thm:main} the rigidity of $\Dtt$ is equivalent to a positive
answer to Question~\ref{q:rep}. For the Turing degrees, Slaman and Woodin
\cite{SW} proved that every automorphism is induced by an arithmetic function.
We do not know whether their analysis can be carried out for the truth-table
degrees. What is known in this direction is Anderson's theorem \cite{And} that
every automorphism of $\Dtt$ is the identity on a cone.

\subsection*{Acknowledgments}
This article was produced with assistance from Claude, an AI model made by Anthropic.
The proof of the main results was formalized in the proof assistant Lean 4 using Aristotle (Harmonic), see \url{https://github.com/bjoernkjoshanssen/tt-auto}.
The author served as proof engineer, and thanks Andrew Marks for graciously sharing the note
\cite{Mar}.

\end{document}